\documentclass[11pt,letterpaper,reqno]{amsart}
\usepackage{caption}
\usepackage{color} 
\usepackage[utf8]{inputenc}
\usepackage[T1]{fontenc}
\usepackage{amssymb,amsmath,mathrsfs,amsthm}
\usepackage{mathtools}
\usepackage{amsfonts}
\usepackage{dsfont}
\usepackage{color}

\usepackage{esint,comment}

\allowdisplaybreaks
\newtheorem{thm}{Theorem}

\newtheorem{corollary}[thm]{Corollary} 

 \newtheorem{lemma}[thm]{Lemma}

\theoremstyle{definition}
\newtheorem{defn}[thm]{Definition}
\newtheorem{rem}{Remark}

\def \no#1#2#3 {{\bf #1} (#3), #2.}
\def \eds#1#2#3 {#1, #2, #3.}

\def\R{{\mathbb R}}

\def\d{{\rm d}}

\def\N{{\mathbb N}}

\def\:{{\colon}}

\def\be#1{\begin{equation}\label{#1}}
\def\ee{\end{equation}}

\def\<{\langle}
\def\>{\rangle}
\def\coloneqq{:=}

\newcommand{\na}{\nabla}

\newcommand{\wo}{\widetilde{\omega}}

\newcommand{\lec}{\lesssim}
\newcommand{\gec}{\gtrsim}
\newcommand{\bs}{\begin{split}}
\newcommand{\essss}{\end{split}}

\renewcommand{\lec}{\lesssim}

\renewcommand{\div}{\operatorname{div}}
 
\newcommand{\eqnb}{\begin{equation}}
\newcommand{\eqne}{\end{equation}}

\renewcommand{\ee}{\mathrm{e}}

\newcommand{\p}{\partial}

\renewcommand{\tt}{\tau}

\newcommand{\im}{\mathrm{Im}}

\renewcommand{\R}{\mathbb{R}}

\newcommand{\CC}{\mathbb{C}}

\renewcommand{\d}{\mathrm{d}}

\newcommand{\supp}{\operatorname{supp}}
\newcommand{\diam}{\operatorname{diam}}

\newcommand\blfootnote[1]{%
  \begingroup
  \renewcommand\thefootnote{}\footnote{#1}%
  \addtocounter{footnote}{-1}%
  \endgroup
}

\begin{document}
\title[Instantaneous gap loss of regularity for 2D Euler]{Instantaneous gap loss of Sobolev regularity for the 2D incompressible Euler equations}
\author{Diego C\'ordoba,  Luis Mart\'inez-Zoroa,  Wojciech S. O\.za\'nski} 
\maketitle
\blfootnote{\noindent D.~C\'ordoba: Instituto de Ciencias Matem\'aticas, 28049 Madrid, Spain, email: dcg@icmat.es\\
L.~Mart\'inez-Zoroa: Instituto de Ciencias Matem\'aticas, 28049 Madrid, Spain, email: luis.martinez@icmat.es\\
W.~S.~O\.za\'nski: Department of Mathematics, Florida State University, Tallahassee, FL 32301, USA, \\
and Institute of Mathematics, Polish Academy of Sciences, 00-656 Warsaw, Poland, email: wozanski@fsu.edu}

\begin{abstract}
    We construct solutions of the 2D incompressible Euler equations in $\R^2\times [0,\infty)$ such that initially the velocity is in the super-critical Sobolev space $H^\beta$ for $1<\beta<2$, but are not in $H^{\beta'}$ for $\beta'>1+\frac{(3-\beta)(\beta-1)}{2 - (\beta-1)^2}$ for $0<t<\infty$. These solutions are not in the Yudovich class, but they exist globally in time and they are unique in a~determined family of classical solutions.
    
\end{abstract}

\section{Introduction}

 We consider the incompressible Euler equations 
\begin{eqnarray}\label{Euler}
\p_t v + (v\cdot\nabla) v + \nabla P = 0,\\
\div\, v = 0 \nonumber
\end{eqnarray}
in $\R^d \times \R_+$, with $d=2,3$, where $v(x,t)=(v_1(x,t),.., v_d(x,t))$ is the velocity field and $P=P(x,t)$ is the pressure function. In this paper we study ill-posedness of the initial value problem for (\ref{Euler}) with a given initial data $v_0(x)=v(x,0)$.

In order to illustrate the ill-posedness phenomena, we first note that the classical theory of the Euler equations goes back to the work of Lichtenstein \cite{Lich} and Gunther \cite{Gunther}, who showed local well-posedness in $C^{k,\alpha}$ ($k\geq 1$, $\alpha\in(0,1)$). This was extended to global-in-time well-posedness in the $2$D case  by Wolibner \cite{wol} and H\"older \cite{Hol}.
In the case of Sobolev spaces, Ebin and Marsden \cite{EM} proved, in a~compact domain, local well-posedness in $H^s$ for $s>\frac{d}{2} + 1$, and  Bourguignon and Brezis \cite{BB} have generalized it to the space $W^{s,p}$ for $s>\frac{d}{p} + 1$. Moreover, Kato \cite{K} extended the local well-posedness to $\R^d$ for initial data $u_0$ in $H^s$ for $s>\frac{d}{2} + 1$, see the extension to the $W^{s,p}$ spaces due to Kato and Ponce \cite{KP}.  

Remarkably, in the $2$D case these local-in-time results can be easily extended for all times using the Beale-Kato-Majda criterion \cite{BKM}, since the vorticity is transported by the flow. The optimal bound for growth was obtained by Kiselev and \v{S}ver\'ak \cite{KS} in a~disk, see also the work by Zlato\v{s} \cite{Z} and the lecture notes \cite{Ki} by Kiselev for further results. 

 Moreover, it can be shown that the equations are not well-posed in some spaces, such as integer $C^k$ spaces ($k\geq 1$). This  was recently demonstrated by Bourgain and Li \cite{Bourgaincm}, and independently by  Elgindi and Masmoudi \cite{Elgindi}, who  showed strong ill-posedness and non-existence of uniformly bounded solutions for the initial velocity $v_0$ in $C^k$. Furthermore, nonexistence of uniformly bounded solutions in the critical Sobolev space $H^{\frac{d}{2} + 1}$ was established in another work of Bourgain and Li \cite{Bourgainsobolev}. Subsequently, Elgindi and Jeong  \cite{Elgindisobolev} obtained analogous results with a~different approach, and Jeong \cite{Injee} gave a~simpler proof and similar results for the critical space $W^{s,p}$. Recently, Kwon proved in \cite{Kwon} that there is still strong ill-posedness in $H^2$ for a~regularized version of the 2D incompressible Euler equations. We also refer the reader to Misio\l{}ek and Yoneda \cite{MY} for a~proof of a~nonexistence result in critical Besov spaces in $d=3$.  
 
 These results gave the first methods of studying ill-posedness and nonexistence of solutions to the Euler equations. Moreover, subsequently Elgindi \cite{Elgindi2} proved a~remarkable result on singularity formation of the 3D axisymmetric Euler equations without swirl for $C^{1,\alpha}$ velocity, where $\alpha >0$ is sufficiently small, and Elgindi, Ghoul, and Masmoudi \cite{Elgindi3} extended it to the finite energy case. We also refer the reader to the work of Chen and Hou \cite{Hou}, who provided evidence of a~possibility of nearly self-similar blow near a~boundary, as well as their subsequent impressive work \cite{Hou2}.

In the case of supercritical Sobolev spaces DiPerna and Lions \cite{DPL} show that for $d=3$ and for every $p\geq 1$, there exists a~shear flow solution to (\ref{Euler}) with $v_0\in W^{1,p}$ and $v(x,t)\not\in W^{1,p}$ for $t>0$. Using the structure of shear flows Bardos and Titi \cite{BT} showed the instantaneous loss of smoothness of weak solutions for the $3$D Euler equations with initial data~in the Holder space $C^{\alpha}$ with $\alpha\in(0,1)$. Note that these constructions rely strongly in the $2 + \frac12$ dimensional structure of the shear flows. At this point is worth mentioning the ground-breaking work of De Lellis and Sz\'ekelyhidi Jr. \cite{DL,DL2}, where they show non-uniqueness of solutions in $L^2$ by the method of convex integration (see also the work of Wiedemann \cite{Wi}). Very recently, using similar tools, Khor and Miao \cite{KM} use the method of convex integration to construct infinitely many distributional 3D solutions  in $H^{\beta}$ for $0<\beta << 1$ which has an instantaneous gap loss of Sobolev regularity. \\

From now on in the present work we will focus in solutions with sufficient regularity in the two dimensional case  and  use the vorticity formulation, which is obtained by taking the $\mathrm{curl}$ of the first equation of (\ref{Euler}) and denoting the scalar function (vorticity) by $\omega \coloneqq \mathrm{curl}\,v= \p_1 v_2 - \p_2 v_1 $, where $\p_1$, $\p_2$ denote partial derivatives with respect to $x_1$, $x_2$, respectively. The equation for the vorticity reads

\eqnb\label{vorticity}
\p_t  \omega  + v \cdot \nabla \omega= 0.
\eqne
According to the Biot-Savart law, there is a~stream function $\psi$ such that $v=(-\p_2 \psi , \p_1 \psi )$ and $-\Delta \psi = \omega$
which gives that $v[\omega]= -\Delta^{-1} \nabla^{\perp} \omega$, where $\nabla^{\perp} \coloneqq (-\p_2, \p_1)$. Thus the velocity field $v$ can be expressed as
\begin{equation}\label{bs_law}
    v[\omega](x,t) = \frac{2}{\pi}\int_{\R^2} \frac{(x-y)^{\perp}\omega(y,t)}{|x-y|^{2}}dy
\end{equation}
where $(x_1, x_2)^{\perp} \coloneqq  (-x_2, x_1)$, although we will ignore the factor $\frac{2}{\pi}$ in our computations since both velocities produce the exact same qualitative behaviour.

In \cite{Y} Yudovich proved the existence and uniqueness of weak solutions for bounded vorticity in a~bounded domain. This statement can be extended to $\R^2$ for solutions such that $\omega\in L^1 \cap L^{\infty}$ (see discussions in \cite{MB} and \cite{Camillo+}). Very recently Vishik  \cite{V1,V2} showed that, although there is existence of solutions with a~force source, the uniqueness fails if $L^{\infty}$ is substituted by $L^{p}$ with $p < \infty$ (see also \cite{Camillo+}).

The main result in this paper is to construct unique solutions of the 2D incompressible Euler equations (in vorticity formulation) in $\R^2\times [0,\infty)$  with initial vorticity in the super-critical Sobolev space $H^\beta$, $0<\beta<1$, which, at each time $t>0$, does  not belong to any $H^{\beta'}$, where  
\eqnb\label{betas}
\beta'>\frac{(2-\beta)\beta}{2 - \beta^2}.
\eqne
Moreover these solutions are not in the Yudovich class but are the unique classical solution in the sense given by Definition~\ref{classol}.

We note that the only result to-date in the  direction of proving instantaneous loss of regularity for 2D Euler in the supercritical regime with velocity $v(t) \in H^1$ for all $t\geq 0$ is the result of Jeong \cite{jeong}, who constructed  solutions to the 2D Euler equations which belong to the Yudovich class but the derivative of the vorticity  loses integrability continuously in time, i.e. $\omega\notin W^{1,p(t)}$, with $p(t)$ decreasing continuously in $t$, $1\leq p(0)<2$. In fact, it is shown in \cite{Elgindisobolev} that for this regularity the solution cannot have  a~jump in the regularity class. Moreover, an instantaneous loss of supercritical Sobolev norms was obtained previously in \cite{Zoroacordoba} for the 2D surface quasi-geostrophic equation (SQG), which is a~more singular active scalar equation. Furthermore, Alberti, Crippa and Mazzucato \cite{ACM} show a~gap loss of Sobolev regularity for a~passive scalar that is driven by a~non-Lipschitz incompressible velocity field, see also \cite{CEIM}. We also note a result of gap loss of Sobolev regularity in the context of the super-critical wave equation, due to Lebeau \cite{Lebeau}.

\subsection{Main results}

We are interested in showing loss of regularity for solutions with vorticity $\omega\in H^{\beta}$, but as the first step we will prove that there are initial conditions $\omega_{0}\in C_c^{\infty}$ that are not big in $H^{\beta}$ but become arbitrarily big in $H^{\beta'}$ for $\beta'$ as in \eqref{betas}.

\begin{thm}[Norm inflation for smooth data.]\label{illpos}
Given $T,K>0$, $\beta\in(0,1)$ and $\beta'>\frac{(2-\beta)\beta}{2-\beta^2}$, there exists a~finite energy initial condition $\omega_{0}\in C_c^{\infty}$ with $\|\omega_{0}\|_{H^{\beta}}\leq 1$ such that the unique global-in-time classical solution $\omega$ to the 2D Euler equations with initial condition $\omega_0$ fulfils $\|\omega\|_{H^{\beta'}}\geq K$ for $t\in [\frac{1}{T},T]$.
\end{thm}

We then consider an infinite number of rapidly growing solutions and use a~gluing argument to find initial conditions that lose regularity instantly.

\begin{thm}[Loss of regularity in the supercritical regime]\label{lossreg}
For any $\epsilon>0$, $\beta\in(0,1)$ there exists a~finite energy initial conditions $\omega_0 $ such that the unique global classical solution $\omega$ to the 2D Euler equations (in the sense of  Definition~\ref{classol}) with such initial condition satisfies 
$$\|\omega_0 \|_{H^{\beta}}\leq \epsilon, $$
$$\|\omega(x,t)\|_{H^{\beta'}}=\infty\quad \text{ for }\quad t \in(0,\infty),\beta' >\frac{(2-\beta)\beta}{2-\beta^2}.$$
\end{thm}

We note that, since the initial condition from Theorem~\ref{illpos} belongs to $ C_c^{\infty}$, for Theorem~\ref{illpos} the solution $\omega$ is defined using the usual definition of classical solutions for the $2D$ Euler equations. However, since Theorem~\ref{lossreg} is concerned with an initial condition with very low regularity, we need to be a~little more precise regarding what we consider a~classical solution to 2D Euler in such a~situation.

\begin{defn}\label{classol}
We say that  $\omega \in L^\infty([0,T); L^{1}\cap L^{p})$, where $p>2$, is a~\emph{classical solution} to 2D Euler with initial conditions $\omega_{0}$ if 
\[
\omega \in C^1_{x,t} (K) \qquad \text{ for every } K= \overline{B(0,d)}\times [0,a] \subset \R^2\times [0,T)
\]
and
\[
\begin{split}
\p_t  \omega + v[\omega]\cdot \nabla \omega &=0\hspace{1.5cm} \text{ in } \R^2 \times [0,T),\\
\omega(x,0)&=\omega_{0}(x)\qquad \text{ for }x\in \R^2.
\end{split}\]  
\end{defn}

Since $\omega$ is $C_{x,t}^1$ on each  compact set this assures that the transport equation is well-defined in the classical sense, that the $L^{p}$ norms are conserved (whenever they are well defined) and that the support of $\omega$ is transported with the velocity $v[\omega]$.

We note that the initial conditions considered will not, in general, belong to the Yudovich class, but instead belong to $L^{1}\cap L^{p}$ for some $p\in (2,\infty) $. Therefore, since local well-posedness of classical solutions is not clear, particularly regarding large times and uniqueness,  we will need to resolve these problems by hand.

\subsection{Ideas of the proof}\label{sec_ideas}

In order to prove the norm inflation result, Theorem~\ref{illpos}, we start by considering $\omega_0$ consisting of a~stationary radial function and a~perturbation involving highly oscillatory angular behaviour, 
\eqnb\label{ansatz1}
\omega_{0}(x)=f(r)+g(r)\frac{\cos (N\alpha)}{N^{\beta}}.
\eqne

We note that, as $N$ grows, the effects of the velocity produced by $g(r)N^{-\beta } \cos (N\alpha)$ become less and less relevant, and thus we can approximate the solution by

$$\p_t  \omega(x,t)+v[f(r)]\cdot \nabla \omega(x,t)=0,$$
see Figure~\ref{fig1} below.\\

\begin{center}
 \includegraphics[width=5cm]{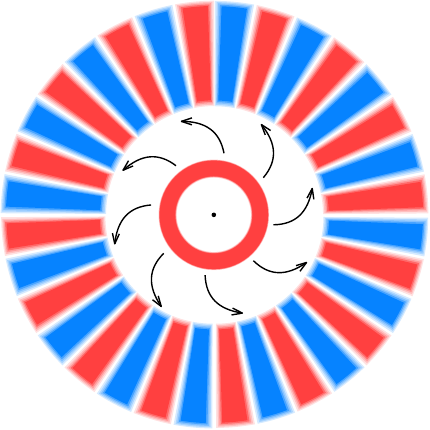}
 \end{center}
 \nopagebreak
 \captionsetup{width=.8\linewidth}
  \captionof{figure}{A sketch of the initial vorticity $\omega_0$. Here the inner vorticity depends only on $r$, and so the resulting velocity field is only angular, which causes rotation of the outer part (here denoted by the arrows). The high frequency $N$ in $\alpha$ of the outer part improves the control over the solution. We note that except for the inner part, the radial part of the vorticity must also include an outer part (supported far from the origin), which would guarantee zero average of $\omega_0$.}\label{fig1}

This already allows us to obtain, in a~fairly straightforward way, strong ill-posedness in $H^{\beta}$, $\beta\in(0,1)$, by choosing $f$ to be small in $H^{\beta}$, but such that $\|v[f(r)]\|_{C^1}$ is large. However, in order to obtain $H^{\beta'}$ norm growth for some $\beta'<\beta$ (as in \eqref{betas}), rather than merely for $\beta'=\beta$, we need to consider a~more general family of initial conditions
\eqnb\label{ansatz2}
\omega_{0} (x) =\omega_{rad} (0) + \omega_{osc} (0) \coloneqq \lambda^{1-\beta}f(\lambda r)+\lambda^{1-\beta} N^{-\beta } g(\lambda r ) \cos ( N\alpha) .
\eqne

Note that such scaling with respect to $\lambda >0$ preserves the $\dot{H}^\beta $ norm. As in \eqref{ansatz1}, the periodicity parameter $N$ allows us to improve our control over the behaviour of the solution and now the scaling parameter $\lambda$ compresses the timescale so that the growth happens faster. The appearance of the new parameter $\lambda$ makes the control of the errors more challenging than in the case of SQG \cite{Zoroacordoba}. We will approximate the solution by a~function of the form
\eqnb\label{pseudo_intro}
\begin{split}
\overline{\omega} (t )& = \overline{\omega_{rad}} (t) + \overline{\omega_{osc}} (t) \\
& \coloneqq \lambda^{1-\beta } f(\lambda r)\\
&\quad +  \frac{\lambda^{1-\beta }}{N^\beta} g(\lambda r) \cos \left( N\left( \alpha-\frac{1}{r}\int_{0}^{t} v_{\alpha}\left[ f(\lambda r)\lambda^{1-\beta}+(error)\right] \d s\right) \right).
\end{split}
\eqne
We note that $\lambda$ is related to $N$ by a~power law, which we describe in  \eqref{lambda_vs_N_intro} below. We note that we will have that $N\gg \lambda$ for $\beta $ close to $0$ and $\lambda \gg N$ for $\beta $ close to $1$. \\

In order to keep track of the regularity of the corresponding solution $\omega(t)$ of the Euler equations \eqref{vorticity} with initial data \eqref{ansatz2}, we first show (in Section~\ref{sec_step1})  that for any $T>0$ we can choose $\lambda$ large enough so that 
\[\omega (t) = \omega_{osc} (t) + \omega_{rad}(t) \qquad \text{ for } t\in [0,T],
\]
where $\omega_{osc}$ and $\omega_{rad}$ remain localized in space. We also show that the influence of $\omega_{osc}$ on $\omega_{rad}$ is exponentially small in $N$, so that $\overline{{\omega}_{rad}}$ approximates $\omega_{rad}$, 
\[
\|\omega_{rad}-\overline{{\omega}_{rad}}\|_{L^{2}}\leq \ee^{-\frac{N}{2}} \qquad \text{ on }  [0,T].
\]
This can be proved by an energy estimate on $W\coloneqq \omega_{rad} - \overline{\omega_{rad}}$, which shows that $\| W \|_{L^2} $ grows exponentially in time of order $\ee^{\lambda^{1-\beta}t}$, as well as by the localization of $\omega_{osc}$ and $\omega_{rad}$, and a~Paley-Wiener-type estimate, which shows that the growth of $\|W \|_{L^2}$ is dominated, on time interval $[0,T]$, by an $O(\ee^{-N})$ smallness of the influence of $\omega_{osc}$ onto $\omega_{rad}$, see Lemma~\ref{locerror} for details. \\

Next, in order to make sure that the evolution of $\omega_{osc}$ is governed, to a~leading order, by $v[\overline{\omega_{rad}}]$ (i.e. that $\omega_{osc}$ can be approximated by $\overline{\omega_{osc}}$), we need to show that $v[\omega_{rad}]$ can be approximated by $v[\overline{\omega_{rad}}]$, and that its effect is not overpowered by $v[\omega_{osc}]$. We address the latter issue by proving that 
\eqnb\label{intro_c1_est}
\| \omega_{osc} (t) \|_{C^1} \leq \lambda^{2-\beta } N^{1-\beta } \exp (C\lambda^{1-\beta })
\eqne
(see Lemma~\ref{winc1}). We then show that $\omega_{osc}$ can be approximated by $\overline{\omega_{osc}}$ by noting that the oscillatory part $\overline{\omega_{osc}}$ of the pseudosolution \eqref{pseudo_intro} satisfies the same PDE as $\omega_{osc}$, except that the velocity field is averaged over $\alpha$, which allows us to use Lagrangian trajectories to show that
\eqnb\label{osc_error_intro}
\| \omega_{osc} - \overline{\omega_{osc}} \|_{L^2} \leq C \lambda^{2-3\beta } N^{-2\beta } \log N.
\eqne
Indeed, the above estimate can be obtained by noting that the radius of the Lagrangian trajectory of $\overline{\omega_{osc}}$ remains constant throughout the flow, as well as using a~version of the classical Log-Lipschitz velocity estimate in polar coordinates \eqref{v_loglip}, a~resulting  $L^\infty$ radial velocity estimate \eqref{vr_NlogN} and the $C^1$ estimate \eqref{intro_c1_est}. \\

We also note that \eqref{osc_error_intro} is the most subtle estimate regarding the relation between $N$ and $\lambda$. Indeed, in order to quantify growth of the solution $\omega_{rad}+\omega_{osc}$ in Sobolev spaces $H^s$, $s\in (0,1)$, we need to control the oscillation error in such spaces in relation to $\| \omega_{osc}\|_{L^2}$ (see \eqref{osc_error} below for details). This can be achieved by first controlling the $L^2$ error directly from \eqref{osc_error_intro}, which gives
\[
\| \omega_{osc} - \overline{\omega_{osc}} \|_{L^2} \lec \| \omega_{osc} \|_{L^2} \lambda^{2-2\beta } N^{-\beta } \log N.
\] 
This shows that, in order to control the oscillation error, we need $N^\beta$ to be a~little bit bigger than $\lambda^{2-2\beta}$, so that we can make sure the error is small and that $\log N$ can be absorbed. Moreover, in order to obtain a~gap for instantaneous loss of regularity that is as large as possible, this in fact dictates the relation between $N$ and $\lambda$. To be more precise we impose the relation
\eqnb\label{lambda_vs_N_intro}
\lambda^{2-2\beta + \delta }=N^\beta,
\eqne
where $\delta >0$ is sufficiently small so that 
\[
\beta_\delta \coloneqq \frac{(2+\delta - \beta )\beta }{2+\delta - \beta^2} > \frac{ (2-\beta ) \beta }{2-\beta^2}.
\]
The last fraction represents the largest gap of instantaneous loss of regularity that can be obtained using our method, in the sense that we will use Sobolev interpolation (in \eqref{growth_needed} below)  to obtain that, for every $\beta'>\beta_\delta$,
\[
\| \omega_{osc} \|_{H^{\beta'}} \geq C\frac{(N\lambda^{2-\beta })^{\beta'}}{(\lambda N)^{\beta}} \geq C \lambda^{\widetilde{\epsilon }}\qquad \text{ for } t\in [1/T,T],
\]
where $\widetilde{\epsilon}>0$ is a~small constant. This which gives the norm inflation claimed by Theorem~\ref{illpos} by taking $\lambda$ sufficiently large.\\

 We note that, in order to obtain the last inequality, one needs to be able to estimate from below the size of the $H^{s}$ norms of the pseudosolution $\overline{\omega }(t)$ for $s\in (0,1)$. While one can use the explicit formula \eqref{pseudo_intro} for the pseudosolution, we note that it is merely ``almost explicit'', which makes the issue nontrivial. 
 
 We show that the error term can be estimated in $C^1$ by a~fractional power of the $C^1$ norm of the leading order term $f(\lambda r )\lambda^{1-\beta }$, but this by itself still does not suggest a~way of computing a~lower bound on $\| \overline{\omega } \|_{H^s}$ using an explicit formula, i.e. the Sobolev-Slobodeckij representation. Instead, we use the Sobolev interpolation $\| \cdot \|_{\dot H^r} \leq c \| \cdot \|^{\frac{r-q}{s-q}}_{\dot H^s}  \| \cdot \|^{\frac{s-r}{s-q}}_{\dot H^q}$, and we choose $r=0$ and $q<0$. This way we can make use of the $L^2$ conservation of $\overline{\omega}$  to obtain a~lower bound, and we need to estimate a~negative Sobolev norm of $\overline{\omega}$ from above. We provide a~subtle argument which gives a~robust estimate of such form, and which can also take into account the error term, see Lemma~\ref{lem_H-delta} for details.\\

As for Theorem~2 we note that taking $\lambda$ larger in the above argument only increases the norm inflation, as well as ensures that it occurs faster and persists for  larger times. Moreover, it also makes the solution more localized. Thus, for each $j$ we can construct a~solution $\omega_j$ to the $2$D Euler equations \eqref{vorticity} such that 
\eqnb\label{pieces_inflation_intro}
\| \omega_j ( \cdot , t ) \|_{H^s} \geq 4^j \qquad \text{ for } s> \frac{(2-\beta ) \beta }{2-\beta^2}+\frac{1}{j}, t\in [4^{-j},1 ],
\eqne
\eqnb\label{pieces_supp_control}
| \supp \, \omega_j | \leq 2^{-j} \quad \text{ for } t\geq 0\text{, with }\quad \supp \, \omega_j \subset B(0,1)\quad \text{ for }t\in [0,2^j]
\eqne
and 
\eqnb\label{pieces_Lp}
\| \omega_j (\cdot , t )\|_{L^p} =C\quad \text{ for all }t\in [0,1], \, p\in [1, 2/(1-\beta ) ] \supset [1,2].
\eqne

Thus considering the rescalings
\eqnb\label{rescalings_gluing_intro}
 \frac{1}{2^j} \omega_{j} \left( x, \frac{t}{2^j} \right), 
\eqne
we obtain the norm inflation of order $2^j$ on time interval $[2^{-j},2^j]$, which expands to $(0,\infty)$ as $j\to \infty$. We can therefore consider a~series of the rescalings \eqref{rescalings_gluing_intro}, translated in the $x_1$ direction by a~rapidly increasing sequence distances $R_j$, defined by $R_0\coloneqq 0$, $R_{j+1} \coloneqq R_j+D_j + D_{j+1}$ for some large $D_j$'s, see Figure~\ref{fig3} below and \eqref{DJ}. Let us denote the corresponding translations of \eqref{rescalings_gluing_intro} by $ \tilde{\omega}_j (x,t)$.

\begin{figure}[h]
\begin{center}
 \includegraphics[width=\textwidth]{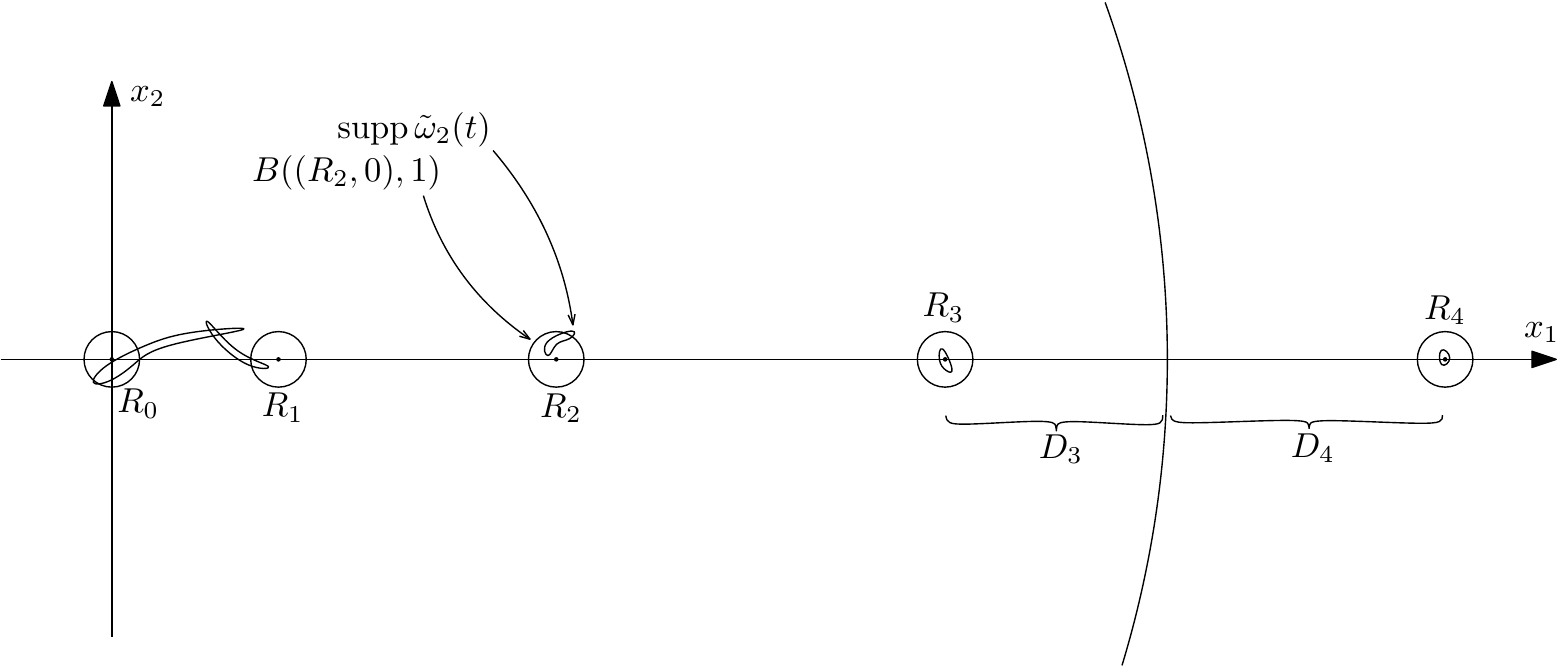}
 \end{center}
 \nopagebreak
 \captionsetup{width=.8\linewidth}
  \captionof{figure}{A~sketch of the gluing argument. This shows the support of the first few individual pieces $\tilde{\omega}_j$ at some time $t\in (2^{-j},2^j)$, where $j=4$. Note that, given $j$ and $t\in[0,2^{j}]$, $\supp \, \tilde{\omega}_{k}(t)\subset B((R_k,0),1)$ for $k\geq j$. } \label{fig3} 
\end{figure}
In order to obtain the claimed gap loss of Sobolev regularity, we first perform a~subtle limiting argument to show existence of a~solution to the $2$D Euler equations \eqref{vorticity} with the corresponding initial data. In fact, we show strong convergence of the classical solution for a~truncated initial condition (i.e. consisting of the first $J$ pieces, $J\geq 0$) in $C^0_tH^4_x (K)$ for any compact set $K\subset \R^2 \times [0,\infty )$, which gives us a~limit $\omega_\infty$ that is a~classical solution in the sense of Definition~\ref{classol} above, see \eqref{limit_in_Linfty_H4} for details.\\

We can then observe that, given $t>0$ and $\beta' > (2-\beta )\beta / (2-\beta^2)$, we can pick a~sufficiently large $j$ so that the norm inflation \eqref{pieces_inflation_intro} implies arbitrarily large $H^{\beta'}$ norm of the $j$-th piece of $\omega_\infty$, and we need to make sure that the pieces do not interact with each other too much to affect this norm inflation.\\ 

To this end we note that the pieces are localized in the sense that, given $t\in [0,2^j )$, the support of $\tilde{\omega}_j$ is contained within $B((R_j,0),1)$. This and \eqref{pieces_supp_control} give us an increasingly better control as $j\to \infty$. On the other hand, for the small values of $j$, we lose the control of the individual pieces (which can, for example, leave $B((R_j,0),1)$ and interact with each other, as sketched in Figure~\ref{fig3}), but the support of all pieces  has measure bounded by $1$ and is included in $B(0,R_j+D_j)$, which implies that it is separated from the further pieces, see Fig.~\ref{fig3} for a~sketch.  
This can be obtained thanks to the $L^p$ norm control \eqref{pieces_Lp}, which implies a~finite maximal speed $v_{max}$, and a~choice of the $D_j$'s (see \eqref{DJ}), as well as the fact that our $C^0_t H^4_x$-loc argument lets us obtain the required properties of our constructed limit $\omega_\infty$.\\

This control of  the distances between pieces of $\omega_\infty$ lets us show that, given $t\in (2^{-j},2^j)$, the norm inflation of the $j$-th piece of $\omega_\infty$ is not affected by either the following pieces or by the sum of the previous pieces,  see \eqref{inflation_gluing} for details. We emphasize that this argument implies not only that $H^{\beta'}$ regularity has lost instantly at $t=0$, but also remains lost for all $t>0$.\\

A~similar argument can be used to show uniqueness of $\omega_\infty$, except that we need to make use of the both properties of the localization: the control of the distances between pieces and the measure of their supports. Moreover, we need to use Lagrange trajectories to keep track of the trajectories of the particles originating from each piece (see \eqref{gluing_uniq_traj}). These facts, together with the $C^1$ bounds of each of the pieces at $t=0$ (see \eqref{pieces_c1_bds_uniq}) and estimates of the Biot-Savart law \eqref{bs_law}, let us estimate the $C^1$ norm of the vorticity evolving from each piece (see \eqref{temp22} for details), given any solution in the sense of Definition~\ref{classol}, and let us establish a~minimal growth of the $R_j$'s (which involves $4$ exponential functions in $j$, see \eqref{DJ_actual}), that allows an $L^2$-based uniqueness proof (see \eqref{l2_uniq_setup1}--\eqref{l2_uniq_setup2} for the main setup). \\

In fact, supposing there are two distinct solutions that coincide until some time $T\geq 0 $, we pick a~$j_0\in \N$ (dependent on $T$) that identifies the piece after which  the uniqueness is unlikely to occur. Namely we pick $j_0$ such that $2^{j_0} \sim T$ (e.g. $j_0=4$ in Fig.~\ref{fig3}), which, for each $j\geq j_0$, allows us to efficiently control the $C^1$ norm of the vorticity originating from the $j$-th piece for each $j \geq j_0 $, using the $C_t^0 H^4_x $ convergence on compact sets of our construction, see \eqref{l2_uniq_setup1}. As a~result we can make the final choice of the initial distances between pieces (see \eqref{DJ_actual}), such that, for each such $j$, the $L^2$ norms of the differences between the $j$-th pieces of the two distinct solutions can be estimated by a~constant that is arbitrarily small with respect to $j$ (see \eqref{l2_uniq_temp}). In order to make the resulting sum convergent, we simply pick $j^{-2} $ (see \eqref{Wj_final}). On the other hand, we apply a~rougher estimate for $j<j_0$ (see \eqref{l2_uniq_setup2}) to obtain  an $L^2$ estimate covering all such pieces at the same time. This gives uniqueness by a~simple argument by contradiction (see \eqref{l2_uniq_concl} for details).

\subsection{Outline of the paper}

In Section~\ref{prel} we give some basic notation that we will use throughout the paper, as well as some preliminary facts. In Section~\ref{estvel} we obtain some technical bounds related to the Biot-Savart law \eqref{bs_law} as well as an upper bound on a~negative Sobolev norm of functions used our construction. In Section \ref{sec_initial} we give the family of initial conditions that allows us to show Sobolev norm inflation and we prove such growth. Finally, in Section \ref{gluing}, we show that a~gluing argument allows us to build a~global in time solution that loses regularity, and we show that it is the unique classical solution with the given initial conditions.

\section{Notation and preliminaries}\label{prel}

Throughout the text we will use functional norms, such as $H^{\beta}$ for example, which refer to the spatial variables. For example, $\|f(x,t)\|_{H^{\beta}}$  refers to the spatial $H^\beta$ norm for the specific time (or times) considered. The only exception to this rule appears in Section \ref{gluing} below, where we prove loss of regularity in a~way that requires different treatment of the space and time regularity. In order to avoid confusion, we will use sub-indexes to indicate the relevant variable for a~norm; for example $\|f(x,t)\|_{C^{1}_{x}}$ would denote the spatial $C^1$ norm (for a~fixed $t$) and $\|f(x,t)\|_{C^{1}_{x,t}}$ would denote the $C^1$ in both space and time.\\

We denote by $\p_t$ the partial derivative with respect to $t$, and by $\p_i$ the partial derivative with respect to $x_i$, $i=1,2$. Throughout the paper we write ``$\lec$'' to denote ``$\leq C$'', where $C>1$ is some universal constant, whose value may change from line to line.  We set $\Lambda\coloneqq (-\Delta )^{1/2}$, and we recall \cite[(4) in Chapter~V]{stein} that
\eqnb\label{formula_for_Lambda_-delta}
\Lambda^{-\delta } f (x) = \frac1{\gamma (\delta )} \int_{\R^2} \frac{f(y)}{|x-y|^{2-\delta}} \d y,
\eqne
where $\gamma (\delta ) \coloneqq \pi 2^\delta \Gamma (\delta /2) /\Gamma (1-\delta /2)$. We will use the following ODE fact: 
\eqnb\label{ode_fact}
\text{If }f'(t) \leq cf(t) + b\text{ and }f(0)=0\text{ then }f(t)\leq \frac{b}{c} \left( \ee^{ct} -1 \right)\leq bt \ee^{ct}.
\eqne

We will make use of polar coordinates, namely, given $(x_{1},x_{2})\in \R^2$, we  define $(r,\alpha)\in [0,\infty )\times (-\pi , \pi ]$ by $x_1 = r \cos (\alpha)$, $x_2 = r\sin (\alpha)$.  Moreover, given $f(r,\alpha)\colon \R^2 \to \R$, we denote by
\[
Af (r) \coloneqq \frac{1}{2\pi }\int_{-\pi }^\pi f(r,\alpha )\,  \d \alpha 
\]
the average of $f$ with respect to $\alpha$.\\

Since most of our computations will be performed in polar coordinates, we will often say that a~function is $\frac{2\pi}{N}-$periodic if, in polar coordinates,  $f(r,\alpha)=f(r,\alpha+\frac{2\pi}{N})$.

Moreover, we will use $v_{r}$ and $v_{\alpha}$ to denote the radial and angular component of the velocity (respectively). For example, \eqref{bs_law} gives that
\eqnb\label{vr_f}
v_r [f ] (x) = \widehat{x}  \cdot  \int_{\R^2}\frac{(x-y)^\perp f (y)}{|x-y|^2 } \d y,
\eqne
where $\widehat{x} \coloneqq x/|x|$. \\

Furthermore, we note that 
\eqnb\label{v_C1_est}
\| v[\omega ] \|_{W^{1,\infty } } \lec (1+R) \| \omega \|_{L^\infty} \left(1+ \log \left( 1+  \| \omega \|_{W^{1,\infty }} \right) \right)
\eqne
for any $\omega \in W^{1,\infty}$ with $\supp\, \omega \subset B(R)$. Indeed, the claim is trivial if  $\na \omega =0$, and otherwise we let $r\coloneqq \|\omega \|_{L^\infty}/\| \na \omega \|_{L^\infty}$ to obtain
\[
\begin{split}
|\na v[\omega ] (x) |&= \frac2\pi \left| \int_{B(x,r) } \frac{(x-y)^\perp \otimes \na \omega }{|x-y|^2} \d y + \int_{B(x,r)^c } \frac{(x-y)^\perp \otimes \na \omega }{|x-y|^2} \d y \right|\\
&\lec \| \na \omega \|_{L^\infty } r +  \| \omega \|_{L^\infty}  + \| \omega \|_{L^\infty } \int_{B(x,r^c)\cap B(0,R)} \frac{\d y}{|x-y|^2} \\
&\lec \| \na \omega \|_{L^\infty } r +  \| \omega \|_{L^\infty}  \left( 1+ \log \frac{R}{r} \right) \\
&\lec \| \omega \|_{L^\infty } \left( 1+ \log(1+R) + \log (1+ \|  \omega \|_{W^{1,\infty }} )\right),
\end{split}
\]
where we integrated by parts in the first inequality and we used the notation $(a\otimes b)_{ij}\coloneqq a_ib_j$. We also have $\|v [\omega ] \|_{L^\infty }\lec (1+R) \| \omega \|_{L^\infty }$, and so \eqref{v_C1_est} follows.

Thus,  if $\omega_0 \in C^\infty_c (\R^2)$, then the unique solution $\omega$ of the Euler equations \eqref{vorticity} satisfies
\[
\begin{split}
\|\omega (t) \|_{C^{1}}&\leq \|\omega_0 \|_{C^{1}} +C\int_0^t \| v[\omega ] \|_{C^1} \| \omega \|_{C^1}  \\
&\leq \|\omega_0 \|_{C^{1}} + C\int_0^t \| \omega \|_{C^1}\|\omega \|_{L^\infty} \log  \| \omega \|_{C^1}  
\end{split}
\]
for every $t>0$ (which can be proved by considering $\| \nabla \omega \|_{L^p}$ and taking $p\to \infty$). Thus, since $\| \omega (t) \|_{L^\infty} \leq \| \omega_0 \|_{L^\infty } \leq \| \omega_0 \|_{C^1}$, we obtain in particular that
\eqnb\label{C2_est_cons}
\|\omega(t)\|_{C^1} \lec \ee^{M\ee^{CMt}},
\eqne
where $M\coloneqq \| \omega_0 \|_{C^1}$.

Finally, we recall the Sobolev-Slobodeckij characterization
\[
\| f \|_{\dot H^s }^2 = C_s \int_{\R^2} \int_{\R^2} \frac{|f(x)-f(y)|^2}{|x-y|^{2+2s}}   \d x \,\d y \qquad \text{ for }s\in (0,1),
\]
see \cite[Proposition~3.4]{DPV} for a proof. In particular, if $\{ f_j \}_j$ is a~family of disjointly supported functions in $\R^2$, then 
\eqnb\label{ss_consequence}
\left\| \sum_j f_j \right\|_{\dot H^s} \geq \sum_j \left\|  f_j \right\|_{\dot H^s}.
\eqne

\section{Velocity and vorticity estimates}\label{estvel}
In this section we study some properties of the vorticity function and the velocity fields given by the Biot-Savart law \eqref{bs_law}. We also estimate $H^\beta$ norms of vorticity functions given in terms of an oscillatory ansatz.

First we note that if $\omega$ is a~smooth solution of the Euler equations \eqref{vorticity} with initial data $\omega_0$, then
\eqnb\label{2pi/N-periodicity}
\omega(t) \text{ is }2\pi/N\text{-periodic for all }t>0\text{ if }\omega_0\text{ is.}
\eqne
Indeed, if $\omega (t)$ is not $2\pi/N$-periodic at any time $t>0$ then $\omega (R_{2\pi/N}x,t)$ is another solution to the Euler equations with the same ($2\pi/N$-periodic) initial data, which contradicts uniqueness, where $R_\alpha $ denotes the rotation operation by $\alpha\in [0,2\pi)$ in $\R^2$.

\subsection{The Log-Lipschitz estimate}

\begin{lemma}[Log-Lipschitz continuity of $v_r$ and $v_\alpha$]\label{lem_v_loglip} Suppose that $\supp\,f\subset \Omega \coloneqq B(0,R)\setminus B(0,R/2)$. Then
\eqnb
|v_r [f ] (x)- v_r [f ] (y)| \lec \| f \|_{L^\infty} |x-y| \left( 1 + \log \frac{R}{|x-y| }\right) 
\eqne
and 
\eqnb
|v_\alpha [f ] (x)- v_\alpha [f ] (y)| \lec \| f \|_{L^\infty } |x-y| \left( 1 + \log \frac{R}{|x-y| }\right) 
\eqne
for any $x,y\in \Omega $.
\end{lemma}

\begin{proof}[Proof of Lemma~\ref{lem_v_loglip}.]
The proof is a~modification of the classical proof (due to Yudovich~\cite{Y}) of the log-Lipschitz bound on $v[f ]$.

For every $x_1,x_2\in \Omega $ we use the formula \eqref{vr_f} for $v_r$ to obtain 
\[
\begin{split}
| v_r [f ]& (x_1)-v_r [f ] (x_2) | \leq \int_{B(x_1,2\delta ) }  \frac{|f (y)|}{|x_1-y |} \d y + \int_{B(x_1,2\delta ) }  \frac{|f (y)|}{|x_2-y |} \d y \\
&\hspace{2cm}+ \int_{\Omega \setminus B(x_1,2\delta ) }  |f (y) | \left| \widehat{x_1} \cdot \frac{(x_1-y)^\perp}{|x_1-y|^2 } -  \widehat{x_2} \cdot \frac{(x_2-y)^\perp}{|x_2-y|^2 } \right| \d y \\
& \lec \| f \|_{L^\infty} \left( \int_{B(2\delta ) }  |y|^{-1} \d y  + \int_{B(3\delta ) } |y |^{-1} \d y \right.\\
&\hspace{2cm}\left.+ \int_{\Omega \setminus B(x_1,2\delta ) }   \left( \frac{ | \widehat{x_1}-\widehat{x_2}|}{|x_1-y| } + \left| \frac{(x_1-y)^\perp}{|x_1-y|^2 } -\frac{(x_2-y)^\perp}{|x_2-y|^2 } \right| \right) \d y \right)  \\
& \lec  \| f \|_{L^\infty} \left(\delta + \int_{\Omega \setminus B(x_1,2\delta ) }  \left( \frac{ \delta }{R |x_1-y| } + \frac{\delta }{|x_* -y |^2} \right)  \d y \right) , 
\end{split}
\]
where $x_*$ is a~point between $x_1,x_2$. We now note that $R\gec |x_1 -y |$ and that $| x_*-y | \sim |x_1 - y|$ to obtain
\[
\begin{split}
| v_r [ f ] (x_1)-v_r [f ] (x_2) |& \lec \| f \|_{L^\infty} \delta\left( 1 + \int_{\Omega \setminus B(x_1,2\delta ) } |x_1 - y |^{-2} \d y \right) \\
&\lec \| f \|_{L^\infty } \delta\left( 1 + \log \frac{R}{\delta }\right) ,
\end{split}\]
as required. An analogous argument gives the same result for $v_{\alpha}$.
\end{proof}

\begin{corollary}\label{loglip}
 Suppose that $\supp\,f\subset \Omega \coloneqq B(0,R)\setminus B(0,R/K)$ for some $K>1$. Then
\eqnb\label{v_loglip}
\begin{split}
|v_r [f ] (x)- v_r [f ] (y)| &+ |v_\alpha [f ] (x)- v_\alpha [f ] (y)| \\
&\quad \lec K \| f \|_{L^\infty } |x-y| \left( 1 + \log \frac{R}{|x-y|}\right) 
\end{split}
\eqne
for all $x,y\in \Omega $.
\end{corollary}
\begin{proof}
The proof follows in the same way as Lemma~\ref{lem_v_loglip} above, except that the inequality $|\widehat{x_1}-\widehat{x_2}|\lec \delta /R$ is replaced by $|\widehat{x_1}-\widehat{x_2}|\lec K \delta /R$.
\end{proof}

This allows us to obtain some improved control over the $L^{\infty}$ bounds of velocities produced by $\frac{2\pi}{N}-$periodic functions.

\begin{lemma}\label{linftyN}
If $\supp \omega \subset \Omega \coloneqq B(0,R) \setminus B(0,R/K) $ for some $K>1$ and $\omega $ is $2\pi/N$-periodic then
\eqnb\label{vr_NlogN}
\| v_r [\omega ] \|_{L^\infty (\Omega )} \lec R \| \omega \|_{L^\infty }N^{-1} \log N.
\eqne
\end{lemma}

Given Corollary \ref{loglip}, we can prove \eqref{vr_NlogN} by first noting that 
\[
A(v_r [ \omega ])=0
\]
for any $\omega$ (by incompressibility). 
Moreover, $\omega $ is $2\pi/N$-periodic, which implies the same for $v_r [\omega ]$. This means that, given $x\in \Omega $ there exists $y\in \Omega $ such that $v_{r}[\omega ] (y) =0$ and $|x-y|\sim C\diam \, (\Omega ) /N$. Thus an application of Corollary \ref{loglip} gives
\eqnb\label{calc2}
\begin{split}
| v_r [\omega ] (x) |& = | v_r [\omega  ] (x)- v_r [\omega ] (y) | \lec \| \omega  \|_{L^\infty} |x-y| \left( 1 + \log \frac{R}{|x-y|} \right) \\
&\lec R \| \omega \|_{L^\infty} N^{-1} \log N,
\end{split}
\eqne
as required.
\subsection{An $\exp(- N)$ decay of the radial velocity of $2\pi/N$-periodic vorticities}
Here we show that a~compactly supported vorticity function that is $2\pi/N$-periodic generates a~velocity field whose radial part decays exponentially fast as $N\to \infty$.
\begin{lemma}\label{lem_exp_decay}
Let $\omega \in L^\infty (\R^2) $ be $2\pi/N$-periodic and such that $\supp\, \omega \subset B(0,a_2)\setminus B(0,a_1)$. Then 
\eqnb\label{vr}
|v_r [\omega ] (r,\alpha ) | \lec \frac{a_2-a_1}r \| \omega \|_{L^\infty } \ee^{- N}
\eqne
for $r\in [ 0, a_1^2/12a_2]\cup [ 12  a_2,\infty )$. 
\end{lemma}

\begin{proof}
First note that if $\omega (r,\alpha ) = g(r) \sin(N\alpha)$ then 
\eqnb\label{vr_formula}
v_r [\omega ] (r,\alpha ) =  \cos (N\alpha ) \mathrm{p.v.} \int_{0}^\infty \int_{-\pi}^\pi (r')^2 \frac{\sin \alpha' g(r' ) \sin (N\alpha' ) }{ (r'-r)^2 + 2r' r (1-\cos \alpha' ) } \d \alpha' \d r' ,
\eqne
 and a~similar formula holds if $\sin(N\alpha)$ is replaced by $\cos(N\alpha)$.

In order to analyze \eqref{vr_formula}, we first consider  
\[
f(z) \coloneqq \frac{\sin z}{B + (1-\cos z )},
\]
where $B>0$, and we note that $f$ is holomorphic in $\CC \setminus \{ x+iy \colon x=2k\pi , y = -\log (1+B\pm \sqrt{B^2+2B } ) \} $ and $2\pi$-periodic in the real direction. Thus, by the Cauchy Theorem,
\[
\left| \int_{-\pi }^\pi f(z) \ee^{iNz} \d z  \right| = \left| \int_{-\pi }^\pi f(i\gamma + z) \ee^{iN(i\gamma +z)} \d z  \right| \leq 2\pi \ee^{-\gamma N } \sup_{\R\times \{ |\im |\leq \gamma \} } |f | ,
\]
where
\eqnb\label{def_of_gamma}
\gamma \coloneqq \frac12  \log (1+B+ \sqrt{B^2+2B } )  ,
\eqne
and we used the fact that $-\log (1+B-\sqrt{B^2+2B})=\log (1+B+\sqrt{B^2+2B})$. Since for $y \in [-\gamma , \gamma ]$ we have $|\cos z| \leq \cosh y \leq \ee^\gamma  \leq \sqrt{1+B+ \sqrt{B^2+2B } }  $, we obtain
\[
|B+1-\cos z| \geq B+1 - \ee^\gamma \geq B+1 - \sqrt{1+B+\sqrt{B^2 +2B }} \geq 1
\]
for such $y$ and $B\geq 5$. Hence $|f| \lec 1 $ for such $y$, $B$. In particular, since also $\gamma \geq 1$ for $B\geq 5$, we obtain that
\eqnb\label{exp_prebound}
\left| \int_{-\pi }^\pi f(x) \sin (Nx ) \d x  \right| \leq 2\pi \ee^{- N}
\eqne
for such $B$.\\

Given $r>0$ we expand $\omega (r,\alpha )$ into Fourier series in $\alpha$. Due to to $2\pi/N$-periodicity we have
\[
\omega (r,\alpha )= g(r,0) + \sum_{k\geq N} \left( g(r,k) \cos (k\alpha ) + h(r,k) \sin (k\alpha ) \right) ,
\]
where 
\[
g(r,k) + i h(r,k) \coloneqq  \frac1\pi  \int_{-\pi}^\pi \omega (r,\alpha ) \ee^{ik\alpha  } \d \alpha, \qquad g(r,0) \coloneqq \frac{1}{2\pi } \int_{-\pi }^\pi  \omega (r,\alpha ) \d \alpha.
\]
Clearly $|g(r,k)|,|h(r,k)|\leq 2 \| \omega \|_{L^\infty}$ for each $r,k$. Moreover, since  $r\in [ 0, a_1^2/12a_2]\cup [12 a_2,\infty )$, a~direct computation shows that  
\[
B\coloneqq \frac{(r'-r)^2}{2r' r} \geq 5
\]
for each $r'\in [a_1  , a_2]$.  Note that the term $g(r,0)$ does not contribute to $v_r[\omega ]$. On the other hand, for each  $k\geq N$, we can apply \eqref{exp_prebound} (and an analogous estimate for $\cos$) to obtain 
\[
\begin{split}
|v_r [\omega ] (r,\alpha )|&\leq \int_{a_1}^{a_2} \frac{r'}{2r} \sum_{k\geq N} \left( \left| \int_{-\pi }^\pi  \frac{\sin \alpha' g(r',k) \cos (k\alpha' )}{B+ 1-\cos \alpha'} \d \alpha'  \right| \right.\\
&\hspace{3cm}+\left. \left| \int_{-\pi }^\pi  \frac{\sin \alpha' h(r',k) \sin (k\alpha' )}{B+ 1-\cos \alpha'} \d \alpha'  \right| \right) \d r'\\
&\lec \| \omega \|_{L^\infty}\int_{a_1}^{a_2} \frac{r'}{2r} \left( \sum_{k\geq N} \ee^{-k} \right)  \d r' \lec \frac{a_2-a_1}{r} \| \omega \|_{L^\infty} \ee^{-N},
\end{split}
\]
as required. 
\end{proof}

\subsection{Sobolev norms for high frequency ansatz}\label{sec_sob_norm_pseudosol}
In this section we prove a~technical lemma that allows us to bound from above a~negative-order homogeneous Sobolev norm of certain functions supported in an annulus in $\R^2$.
\begin{lemma}\label{lem_H-delta}
Given $\epsilon\in (0,1)$, $\delta \in (0,\epsilon )$, and $f\in C^2 ([1/2,4])$ with $f'>0$ in $[1/2,4]$, there exist $C, K_0\geq 1$ such that 
\[
\omega_{K}(r,\alpha)\coloneqq g(r)\cos(N\alpha-Kf(r)+f_{err}(r)) 
\]
satisfies
\[
\| \omega_{K} \|_{\dot{H}^{-\delta}}\leq C K^{-\delta}\| g\|_{C^{1}}
\]
for every $g \in C^2_c ((1/2,4) )$, $K\geq K_0$, $N\in\N$ and $f_{err}\in C^1 ([1/2,4] )$ such that $\| f_{err} \|_{C^1} \leq K^{1-\epsilon }$.
\end{lemma}

\begin{proof}

We first show that for $r\in(\frac14,6)$
\eqnb\label{first_show} 
|\Lambda^{-\delta}\omega_{K}|\leq C K^{-\delta}\|g\|_{C^1}.
\eqne
To this end we note that rewriting the definition \eqref{formula_for_Lambda_-delta} of $\Lambda^{-\delta }\omega_K$ in polar coordinates gives
$$ \Lambda^{-\delta}\omega_{K}(r,\alpha)=C_{\delta}\int_{-\pi}^{\pi}\int_{0}^{\infty} \frac{\omega_{K}(r',\alpha')}{|(r-r')^2+2rr'(1-\cos(\alpha-\alpha'))|^{\frac{2-\delta}{2}}}r'\d r'\,\d\alpha'.$$
Using the change of variables $h=r'-r$, $\tilde{\alpha}=\alpha'-\alpha$, we can estimate the integral over the region $\{ |\alpha' - \alpha |\leq 1/K \}$, by noting that $r,r+h=O(1)$, which implies that
\eqnb\label{101}
\begin{split} &\left|\int_{\alpha-\frac{1}{K}}^{\alpha +\frac{1}{K}}\int_{0}^{\infty} \frac{\omega_{K}(r',\alpha')}{|(r-r')^2+2rr'(1-\cos(\alpha-\alpha'))|^{\frac{2-\delta}{2}}}r'\,\d r'\d \alpha' \right| \\
&\hspace{3cm}\leq C\| g \|_{L^{\infty}}\int_{-\frac{1}{K}}^{\frac{1}{K}}\int_{-\infty}^{\infty} \frac{1}{|h^2+C\tilde{\alpha}^2|^{\frac{2-\delta}{2}}}\d h \,\d\tilde{\alpha}\\
&\hspace{3cm}\leq C\| g \|_{L^{\infty}}\int_{-\frac{1}{K}}^{\frac{1}{K}}\int_{|\alpha|}^{\infty} \frac{1}{|h|^{2-\delta}}\d h\,\d \tilde{\alpha}\\
&\hspace{3cm}\leq C\| g \|_{L^{\infty}}\int_{-\frac{1}{K}}^{\frac{1}{K}} |\tilde \alpha |^{-(1-\delta )} \d\tilde \alpha\leq C \| g \|_{L^{\infty}}K^{-\delta},
\end{split}
\eqne
as claimed. 

As for $|\tilde \alpha |>1/K$, we first consider $h\in [0,4 ]$ and we divide this interval into $O(K)$ pieces of the form $[a,a+2\pi/(K f'(a+r))]$ and integrate by parts on each of them. Namely, given $a\in [0,4]$ we set 
\[
\begin{split}
u(\tilde h) &\coloneqq \int_a^{\tilde h} (r+h) g(r+h)\cos(N\alpha'-Kf(r+h)+f_{err}(r+h)) \d h,\\
v (h) &\coloneqq {|h^{2}+2r(r+h)(1-\cos \tilde{\alpha})|^{-\frac{2-\delta}{2}}},
\end{split}
\]
so that 
\[
\begin{split}
|v'(h)|&\lec \frac{| 2h +2r (1-\cos \tilde \alpha  )|}{ \left| h^{2}+2r(r+h)(1-\cos \tilde{\alpha} ) \right|^{\frac{4-\delta }{2}}} \lec \left| h^{2}+2r(r+h)(1-\cos \tilde{\alpha} ) \right|^{-\frac{3-\delta }{2}} ,\\
u'(h)&=(r+h)g(r+h)\cos(N\alpha'-Kf(r+h)+f_{err}(r+h)),
\end{split}
\]
$u(a)=0$, and we can estimate $u(h)$ for each $h\in (a,a+2\pi /(K f' (a+r))]$ by the brutal bound
\[
|u(h)| \lec K^{-1} \| g \|_{L^\infty} .
\]
This gives that

\[
\begin{split}
    &\left| \int_{a}^{a+\frac{2\pi}{Kf'(a+r)}} \frac{g(r+h)\cos(N\tilde{\alpha}-Kf(r+h)+f_{err}(r+h))}{|h^2+2r(r+h)(1-\cos(\tilde{\alpha}))|^{\frac{2-\delta}{2}}}(r+h)\d h\right|\\
    &=\left| \int_{a}^{a+\frac{2\pi}{Kf'(a+r)}} u'(h)v(h) \d h  \right| \\
    &\lec \frac{\| g \|_{L^\infty } }{K}  \int_{a}^{a+\frac{2\pi}{Kf'(a+r)}} |v'(h) | \d h  \\
    &\qquad+ v\left( {a+\frac{2\pi}{Kf'(a+r)}}\right) \left|u\left({a+\frac{2\pi}{Kf'(a+r)}}\right)\right|  \\
    &\lec \frac{\|g\|_{L^{\infty}}}{K}\int_{a}^{a+\frac{2\pi}{Kf'(a+r)}} {|h^2+C \tilde{\alpha}^2|^{-\frac{3-\delta}{2}}}\d h+\frac{ \|g\|_{C^{1}}K^{-1-\epsilon}}{|(a+\frac{2\pi}{Kf'(a+r)})^2+C\tilde{\alpha}^2|^{\frac{2-\delta}{2}}},
    \end{split}
    \]
where we used the fact that $1-\cos \tilde \alpha \gtrsim (\tilde \alpha )^2$ as well as the fact that
\[
\begin{split}
&\left|u\left({a+\frac{2\pi}{Kf'(a+r)}}\right)\right|\\
 &=\left|\int_{a}^{a+\frac{2\pi}{Kf'(a+r)}} \Big( g(r+h)\cos(N\alpha' -Kf(r+h)+f_{err}(r+h))\right.\\
    &\hspace{1cm}\left.- g(r)\cos (N\alpha' -Kf(r+a)-hKf'(r+a)+f_{err}(r+a))\Big)\d h\right| \\
   & \lec \| g\|_{C^{1}}K^{-1-\epsilon},
\end{split}
\]
by adding and subtracting the mixed terms, and noting that the difference of the $g$'s gives $C\| g \|_{C^1} K^{-2}$, the second order Taylor expansion of $f$ gives the bound $C\| g \|_{L^\infty }K^{-2}$, and the assumption on $f_{err}$ gives $C\| g \|_{L^\infty } K^{-1-\epsilon }$.

Thus, letting $a\coloneqq h_i$, where $h_{0}\coloneqq 0$, $h_{i+1}\coloneqq h_{i}+\frac{2\pi}{Kf'(h_{i}+r)}$ for $i=0,\ldots , i_0$, where $i_{0}$ is the largest integer such that $h_{i_{0}}\leq 4$, we obtain that $h\in (h_i,h_{i+1})$ for some $i\in \{ 0, \ldots , {i_0} \}$ whenever $r+h \in \supp\, g \cap [r,\infty )$, and
\[
\begin{split}
    &\left| \int_{0}^{4} \frac{g(r+h)\cos(N\alpha' -Kf(r+h)+f_{err}(r+h))}{|h^2+2r(r+h)(1-\cos\tilde{\alpha})|^{\frac{2-\delta}{2}}}(r+h)\d h\right| \\
    &\hspace{1cm}\lec \sum_{i=0}^{i_{0}}\left( \frac{\|g\|_{L^{\infty}}}{K}\int_{h_{i}}^{h_{i+1}} \frac{1}{|h^2+C \tilde{\alpha}^2|^{\frac{3-\delta}{2}}}\d h+\frac{ \|g\|_{C^1}  K^{-1-\epsilon}}{|h_{i+1}^2+C\tilde{\alpha}^2|^{\frac{2-\delta}{2}}}\right)\\
    &\hspace{1cm}\lec \frac{\|g\|_{L^{\infty}}}{K}\int_{0}^{4+\frac1K} \frac{1}{|h^2+C \tilde{\alpha}^2|^{\frac{3-\delta}{2}}}\d h+\frac{\|g\|_{C^1}K^{-\epsilon}}{|\tilde{\alpha}|^{1-\delta}}
\end{split}
\]
and a~similar computation can be done for $h\in (-(r-1/8) ,0]$, which allows us to cover $r+h\in \supp\, g \cap (0,r) $. With this, in particular

\eqnb\label{102}
\begin{split}
    &\left|\int_{\pi\geq |\tilde{\alpha}|\geq\frac{1}{K}} \int_{-r+\frac18}^{4} \frac{g(r+h)\cos(N\alpha' -Kf(r+h)+f_{err}(r+h))}{|h^2+2r(r+h)(1-\cos(\tilde{\alpha}))|^{\frac{2-\delta}{2}}}(r+h)\d h \d\tilde \alpha \right|\\
    &\hspace{1cm}\lec \int_{\pi\geq |\tilde{\alpha}|\geq\frac{1}{K}}\left(\frac{\|g\|_{L^{\infty}}}{K}\int_{0}^{4+\frac1K} \frac{1}{|h^2+C \tilde{\alpha}^2|^{\frac{3-\delta}{2}}}\d h+\frac{\|g\|_{C^1}K^{-\epsilon}}{|\tilde{\alpha}|^{1-\delta}} \right)\d \tilde{\alpha}\\
    &\hspace{1cm}\lec \|g\|_{C^1}\int_{\pi\geq |\tilde{\alpha}|\geq\frac{1}{K}} \left(\frac{1}{K|\tilde{\alpha}|^{2-\delta}}+\frac{1}{K^{\epsilon}|\tilde{\alpha}|^{1-\delta}}\right)\d\tilde{\alpha}\\
    &\hspace{1cm}\lec \|g\|_{C^1}(K^{-1}K^{1-\delta}+K^{-\epsilon})\\
    &\hspace{1cm}\lec \|g\|_{C^1}K^{-\delta}.
\end{split}
\eqne

Next we need to show some bounds for $r\leq \frac14$ and $r\geq 6$. For $r\in (0,1/4)$ we need $h\in (1/4,4)$ (so that $r+h \in \supp \, g$). Thus letting $h_{0}\coloneqq \frac{1}{4}$, $h_{i+1}\coloneqq h_{i}+\frac{2\pi}{Kf'(r+h_i )}$ and letting  $i_{0}\in \N $ be the largest integer such that $r+h_{i_{0}}\leq 4$, and applying integration by parts as before, we have
\[
\begin{split}
    &\left| \int_{h_{i}}^{h_{i+1}} \frac{g(r+h)\cos(N\alpha' -Kf(r+h)+f_{err}(r+h))}{|h^2+2r(r+h)(1-\cos \tilde{\alpha})|^{\frac{2-\delta}{2}}}(r+h)\d h\right|\\
    &\lec \frac{\|g\|_{L^{\infty}}}{{K}}\int_{h_{i}}^{h_{i+1}} \frac{1}{|h^2+Cr \tilde \alpha^2|^{\frac{3-\delta}{2}}}\d h'\\
    &\hspace{0.1cm}+\left( \frac{1}{|h_{i+1}^2+Cr\tilde{\alpha}^2|^{\frac{2-\delta}{2}}}\int_{h_{i}}^{h_{i+1}} g(r')\cos(N\alpha'-Kf(r+h)+f_{err}(r+h))\d h\right)\\
    &\lec \frac{\|g\|_{L^{\infty}}}{K}\int_{h_{i}}^{h_{i+1}} \frac{1}{|h^2+Cr \tilde{\alpha}^2|^{\frac{3-\delta}{2}}}\d h +\frac{\|g\|_{C^1}K^{-1-\epsilon}}{|h_{i+1}^2+Cr\tilde{\alpha}^2|^{\frac{2-\delta}{2}}}.\end{split}
    \]
Thus, summing in $i$, and integrating in $\tilde \alpha \in \{ |\tilde \alpha | \in (1/K,\pi ) \}$ (recall that $\tilde\alpha = \alpha'-\alpha$), we obtain 
\eqnb\label{103} |\Lambda^{-\delta}(\omega_{K})(r,\alpha)|\lec\|g\|_{C^1}(K^{-1}+K^{-\epsilon})\lec \|g\|_{C^1} K^{-\delta}\quad \text{ for }r\in (0,1/4).\eqne

Similarly, for $r\in(6,\infty)$ we need $h\in (r-1/2,r-6)$, which gives the final bound of the form
\eqnb\label{104} |\Lambda^{-\delta}(\omega_{K})(r,\alpha)|\lec \|g\|_{C^1}\left( \frac{K^{-1}}{(r-4)^{3-\delta}}+\frac{K^{-\epsilon}}{(r-4)^{2-\delta}}\right) \lec \|g\|_{C^1} \frac{K^{-\delta}}{(r-4)^{2-\delta}}\eqne
for $r>6$.

Integrating the squares of the above pointwise estimates \eqref{101}--\eqref{104} on $\Lambda^{-\delta } \omega_K$ gives the claimed $L^2$ bound.
\end{proof}

\section{Initial conditions and growth for smooth functions}\label{sec_initial}

Here we prove Theorem~\ref{illpos}, that is we fix $\beta \in (0,1)$, $\beta'> (2-\beta )\beta/(2-\beta^2)$, and $K,T>0$ and we will construct $\omega_0 \in C_c^\infty (\R^2)$ such that $\| \omega_0 \|_{H^\beta } \leq 1$ and that the unique classical solution $\omega$ to the Euler equations admits growth $\| \omega \|_{H^{\beta'}} \geq K $ for $t\in [1/T,T]$.

To this end, we fix $\delta>0$ sufficiently small so that 

\eqnb\label{beta_delta}
\beta_\delta \coloneqq  \frac{(2+\delta-\beta )\beta }{2+\delta  - \beta^2} > \frac{(2-\beta)\beta }{2-\beta^2}
\eqne
satisfies $\beta_\delta < \beta'$.

We will consider radial functions $f(r),g(r)$ such that $g\in  C_c^\infty (\frac12,4)$,  $f\in C_c^\infty ((a,b)\cup(c,d))$  and 
\begin{itemize}
    \item $a\leq 10^{-4}$, $b\in (a,1/16)$, $d\geq 10^{3}$, $c\in (16,d)$, 
    \item $ \p_r \frac{ v_{\alpha}[f](r)}{r} \in(\frac{1}{M},M)$ for some $M>1$ when $r\in(\frac12,4)$, 
    \item $\| f \|_{H^1},\| g \|_{H^1} \leq 1/20$,
    \item $\int_{0}^{\infty}f(r)r\,\d r=0$.
\end{itemize}

A~function $g$ fulfilling the requirement is trivial to obtain, but we need to justify that $f$ with the required properties exists. For this, we first consider some arbitrary positive $\tilde{f}(r)\in C_{c}^{\infty} (10^{-5},10^{-4})$ and we study the quantity $\partial_{r}\frac{ v_{\alpha}[\lambda^2 \tilde{f}(\lambda \cdot )](r)}{r}$  for $r\in(\frac12,4)$.
We observe that
$$v_{\alpha}[\lambda^2 \tilde{f}(\lambda \cdot)](r)=\int_{-\pi}^{\pi}\int_{0}^{\infty}r'\frac{\lambda^2\tilde{f}(\lambda r')(r-r'\cos \alpha)}{ r^2+(r')^2- 2rr'\cos \alpha  }\d r'\,\d \alpha,$$
and so, due to the location of the support of $\tilde{f}(\lambda r)$ we have $v_{\alpha}(\tilde{f}(\lambda r))\geq 0$. Furthermore, for $r\in(\frac12,4)$,
\begin{align*}
    &\p_{r}v_{\alpha}[\lambda^2 \tilde{f}(\lambda \cdot)](r)\\
    &=\int_{-\pi}^{\pi}\int_{0}^{\infty}r'\tilde{f}(\lambda r')\left(\frac{\lambda^2}{ r^2+(r')^2- 2rr'\cos \alpha }\right.\\
    &\hspace{6cm}\left.-\frac{2\lambda^2(r-r'\cos \alpha)^2}{( r^2+(r')^2- 2rr'\cos \alpha )^2}\right)\d r'\,\d \alpha\\
\end{align*}
and hence
$$\lim_{\lambda\rightarrow \infty}\p_{r}v_{\alpha}[\lambda^2 \tilde{f}(\lambda \cdot)](r)=-\frac{2\pi}{r^2}\int_{0}^{\infty} \tilde{f}(s)s\,\d s.$$
Thus,  $\p_{r}v_{\alpha}[\lambda^2 f(\lambda \cdot)](r)<0$ for sufficiently large $\lambda$, and 
$$\p_r \frac{v_{\alpha}[\lambda^2 \tilde{f}(\lambda \cdot)](r)}{r}<0$$
for $r\in (\frac12,4)$, which implies that $-\lambda^2 \tilde{f}(\lambda r)$  gives us the desired effect on the velocity for $\lambda$ big, but this $f$ would clearly  have nonzero average. To compensate for that, we now consider $\lambda^{-2}\tilde{f}( \frac{r}{\lambda})$ for $\lambda\geq 10^{8} $. It is easy to check that, for any $r\in (\frac12,4)$,
\[ v_{\alpha}\left[ \lambda^{-2} \tilde{f}\left( \frac{\cdot}{\lambda}\right) \right] (r)\rightarrow 0\qquad \p_{r}v_{\alpha}\left[ \lambda^{-2} \tilde{f}\left( \frac{\cdot}{\lambda}\right)\right](r)\rightarrow 0
\]
as $\lambda\rightarrow \infty$, so that
$$-\lambda^2 \tilde{f}(\lambda \cdot)+\lambda^{-2} \tilde{f}( {\cdot}/{\lambda})$$
has the desired properties for the velocity and average value for sufficiently large $\lambda$.
 Then, multiplication by some small constant $c>0$ allows us to make the $H^{1}$ norm as small as we want.\\

Having fixed $f$, $g$, from now on any quantity depending on these two functions will be treated as a~universal constant. \\

Given $\lambda>0$ we now set 
\eqnb\label{initial_data}
\omega_{0} \coloneqq  \lambda^{1-\beta } f(\lambda r)+\lambda^{1-\beta } g(\lambda r) N^{-\beta } \cos (N\alpha ) ,
\eqne
 where  $N$ is related to $\lambda$ via 
\eqnb\label{lambda_vs_N}
\lambda^{2-2\beta + \delta } = N^\beta .
\eqne

Note that in particular $\|\omega_{0}\|_{H^{\beta}}\leq 1$ for any $\lambda\geq 1$. We denote by $\omega \colon \R^2\times [0,\infty)\rightarrow \R$ the unique solution of \eqref{vorticity} with initial data $\omega_{0}$.\\

We will use the pseudosolution
\[
\overline{\omega} \coloneqq \overline{\omega_{rad}} + \overline{\omega_{osc}}  ,
\]
where
\eqnb\label{pseudosol_defs}
\begin{split}
\overline{\omega_{rad}} &\coloneqq  \lambda^{1-\beta}f(\lambda r),\\
\overline{\omega_{osc}} &\coloneqq  \lambda^{1-\beta } g(\lambda r) N^{-\beta }\cos \left( N\left( \alpha-\frac{1}{r}\int_{0}^{t} v_{\alpha}\left[ A\omega \right] \d s\right) \right),
\end{split}
\eqne

to approximate the evolution of $\omega$.  We note that $\overline{\omega_{osc}} +\overline{\omega_{rad}}$ is the pseudosolution \eqref{pseudo_intro}, which was discussed heuristically in the introduction (Section~\ref{sec_ideas}). We also observe that  that $\overline{\omega_{rad}}$ is a stationary radial function and that $\overline{\omega_{osc}}$ depends on the solution itself, but only via its angular average. \\

Before we can control the difference $\omega - \overline{\omega}$ and show the rapid growth of $\overline{\omega}$, and so also of $\omega$, we need to prove some basic properties of $\omega$, which we discuss in Steps 1--3 below. We will keep in mind that $\lambda>0$ is a~large parameter that will be fixed in Step~4 (Section~\ref{sec_step4}), where we will prove the growth of the $H^{\beta'}$ norm.

\subsection{Step 1 - localization and control of $\omega_{rad}$}\label{sec_step1}

We decompose $\omega$ into two parts, one that is mostly composed of highly oscillatory terms $\omega_{osc}$ and one that remains mostly radial $\omega_{rad}$. Namely, if $\phi(x,t)$ is the flow map given by $v[\omega]$ we define
\[\begin{split}\omega_{rad}(x,t)&\coloneqq \omega_{rad}(\phi^{-1}(x,t),0),\qquad \omega_{rad}(x,0)=\lambda^{1-\beta}f(\lambda r),\\
\omega_{osc}(x,t)&\coloneqq \omega_{osc}(\phi^{-1}(x,t),0),\qquad \omega_{osc}(x,0)=\lambda^{1-\beta } g(\lambda r) N^{-\beta } \cos (N\alpha) .
\end{split}\]

Note that, with those definitions, we indeed have that

\eqnb\label{om_decomp}
\omega (t) = \omega_{rad} (t) + \omega_{osc} (t). 
\eqne

We now show that these two parts barely change their support and, furthermore, $\omega_{rad}$ stays almost stationary.
\begin{lemma}\label{locerror}
For sufficiently large  $\lambda$
and $t\in [0,T]$, we have
\eqnb\label{supps_localized}
\begin{split}\supp\, \omega_{osc} &\subset B(0, 6\lambda^{-1}) \setminus B(0, (4\lambda)^{-1}),\\
 \supp\, \omega_{rad} &\subset B(0, 2b\lambda^{-1}) \setminus B(0, a(2\lambda)^{-1})\cup  B(0, 2d\lambda^{-1}) \setminus B(0, c(2\lambda )^{-1}),
 \end{split}
\eqne
and furthermore
\begin{equation}\label{w_decomp}
    \|\omega_{rad}-\overline{{\omega}_{rad}}\|_{L^{2}}\leq \ee^{-\frac{N}{2}}.
\end{equation}
\end{lemma}

\begin{proof}
We first note that the claim of the lemma is valid at least for small times.
In order to obtain the localization \eqref{supps_localized} we note that $\omega$ remains $2\pi/N$-periodic for all times, and so we can use Lemma~\ref{linftyN} to deduce that for $t$ such that (\ref{supps_localized}) is satisfied, we have, for $r\in B(0, 2b\lambda^{-1})$ 
$$|v_{r}[w](r)|\leq C (N\lambda)^{-\beta}\log N $$
and, using the relationship \eqref{lambda_vs_N} between $N$ and $\lambda$,
$$|v_{r}[w](r)|\leq C \lambda^{-2+\beta-\delta}\log N.$$
Thus \eqref{supps_localized} remains valid at least for $t\in [0, C \lambda^{1-\beta +\delta } (\log \lambda )^{-1}]$, and so taking $\lambda $ large ensures that \eqref{supps_localized} holds until $T$.\\

As for \eqref{w_decomp}, we note that, since $a \leq 10^{-4}$, $d\geq 10^{3}$, the assumption of Lemma~\ref{lem_exp_decay} holds  and $r\sim 1$ on $\supp\,\omega_{rad}$, and thus
\[
\| v_r [\omega_{osc}] \|_{L^\infty (\supp\, \omega_{rad})} \lec \| \omega_{osc} \|_{L^\infty } \ee^{-N} \leq \| \omega_{osc,0}  \|_{L^\infty} \ee^{-N}.
\]

Moreover, since
$$\p_t  \omega_{rad}+v(\omega_{rad})\cdot\nabla \omega_{rad}+ v(\omega_{osc})\cdot\nabla (\omega_{rad})=0$$
letting $W\coloneqq \omega_{rad}-\overline{{\omega}_{rad}}$ we see that 

$$\p_t W+ v[W]\cdot\nabla W+ v[W]\cdot \nabla \overline{{\omega}_{rad}} +v[\overline{{\omega}_{rad}}]\cdot\nabla W+ v[\omega_{osc}]\nabla(W+\overline{{\omega}_{rad}})=0,$$
 where we have also made the observation that $v[\overline{{\omega}_{rad}}] \cdot \nabla \overline{{\omega}_{rad}} =0$. This  gives us an evolution inequality for the $L^2$ norm

\[\begin{split}
\frac{\d }{\d t}\|W\|_{L^{2}}&\lec \|v[W]\|_{L^{2}}\lambda^{2-\beta}+\ee^{-N} \frac{\lambda^{3-2\beta}}{N^{\beta}}\lec  \|W \|_{L^{2}}\lambda^{1-\beta}+\ee^{- N} \frac{\lambda^{3-2\beta}}{N^{\beta}} ,
\end{split}
\]
where we used that $\|v[W]\|_{L^{2}}\leq \frac{C}{\lambda} \|W\|_{L^{2}}$, since $W$ is supported in a~disc of radius $\frac{2b}{\lambda}$. In light of the ODE fact \eqref{ode_fact}, this gives 
\eqnb\label{wout_L2}
\| W (t) \|_{L^2} \leq C\ee^{\lambda^{1-\beta}t- N}\frac{\lambda^{2-\beta}}{N^{\beta}}
\eqne
which proves the second claim by taking $\lambda$ large.
\end{proof}
Having obtained the decomposition $\omega = \omega_{rad}+\omega_{osc}$ we will use the notation
\eqnb\label{notation_omega_errors}
\begin{split}
\omega_{rad,err} &\coloneqq \omega_{rad} - \overline{\omega_{rad}},\\
\omega_{osc,err} &\coloneqq \omega_{osc} - \overline{\omega_{osc}}.
\end{split}
\eqne

\subsection{Step 2 - $L^\infty $ control of $\nabla \omega_{osc}$} \label{sec_step2}
The $H^{s}$ growth for our solutions will come from the effect of the velocity generated by $\omega_{rad}$ acting on $\omega_{osc}$. However, we need to prove that this effect is not overpowered by the velocity generated by $\omega_{osc}$. For that, we have the following lemma.

\begin{lemma}\label{winc1}
For sufficiently large $\lambda$
\eqnb\label{win_C1}
\| \omega_{osc} (t) \|_{C^1 } \leq  \lambda^{2-\beta } N^{1-\beta } \exp( C \lambda^{1-\beta }  ) \qquad \text{ for } t\in [0,T].
\eqne
\end{lemma}
\begin{proof}
We first note that for any $C>1$  \eqref{win_C1} holds for some short time interval, say for $t\in [0,t_0 ]$.  Moreover, observe that, for $t\in[0,t_{0}],$  
$$\p_t \omega_{osc} + v[\omega_{osc}]\cdot\nabla \omega_{osc}+ v[\omega_{rad} ] \cdot\nabla \omega_{osc}=0,$$
and so, since the localization \eqref{supps_localized} of the radial and oscillatory parts of $\omega$ implies that
\[
\| v [\omega_{rad} ] \|_{W^{1,\infty } (\supp\,\omega_{osc} )} \lec  \mathrm{dist} (\supp\, \omega_{rad} , \supp\, \omega_{osc} )^{-2} \| \omega_{rad} \|_{L^1} \lec \lambda^{1-\beta},
\]
we obtain
\begin{align*}
    \frac{\d \|\omega_{osc}\|_{C^1}}{\d  t}&\leq C(\|v[\omega_{osc}]\|_{C^1}+\lambda^{1-\beta})\|\omega_{osc}\|_{C^1}\\
    &\leq C\left( \|\omega_{osc}\|_{L^{\infty}}\log \|\omega_{osc}\|_{{C^1}}+\lambda^{1-\beta}\right) \|\omega_{osc}\|_{{C^1}}\\
        &\leq C\left( \frac{\lambda^{1-\beta}}{N^{\beta}}\log \|\omega_{osc}\|_{{C^1}} +\lambda^{1-\beta}\right)\|\omega_{osc}\|_{{C^1}}\\
        &\leq C\lambda^{1-\beta}\|\omega_{osc}\|_{{C^1}},
\end{align*}
where we used the $C^1$ velocity estimate \eqref{v_C1_est} in the second line, the $L^\infty$ conservation of the vorticity in the third line, and the assumed bound of the $C^1$ norm in the last line.  Thus 
\[
\|\omega_{osc}(t)\|_{C^1}\leq \|\omega_{osc}(0)\|_{C^1}\ee^{C\lambda^{1-\beta}t}\leq {\lambda^{2-\beta}}N^{1-\beta}\ee^{C\lambda^{1-\beta}}
\]
for all $t\in [0,t_0]$, and a~continuity argument completes the proof of \eqref{win_C1}. 
\end{proof}

\subsection{Step 3 - $L^2$ control of $\omega_{osc,err}$}\label{sec_step3} In this section we show that 
\[
\overline{\omega_{osc}} (r,\alpha ,t )= \lambda^{1-\beta } g(\lambda r) N^{-\beta }\cos \left( N\left( \alpha-\frac{1}{r}\int_{0}^{t} v_{\alpha}\left[ A\omega \right] \d s\right) \right)
\]
(recall \eqref{pseudosol_defs}),
is actually a~good approximation (in $L^{2}$) of $\omega_{osc}$. Note that the pseudosolution $\overline{\omega_{osc}}$  corresponds to $\omega_{osc}$ advected with an averaged velocity, i.e.
\[
\p_t \overline{{\omega}_{osc}}+ v[A{\omega_{osc}}]\cdot\nabla \overline{{\omega}_{osc}}+ v[A{\omega_{rad}}]\cdot\nabla \overline{{\omega}_{osc}}=0.
\]
We now use this fact to obtain the following.

\begin{lemma}\label{winerror}
For $\lambda$ big enough we have
\eqnb\label{inner_error}
\|\omega_{osc,err} (t)\|_{L^{2}}\leq C\frac{1}{(\lambda N)^{\beta}}\frac{\lambda^{2-2\beta}\log N}{N^{\beta}}\leq C\|\overline{{\omega}_{osc}}(t)\|_{L^{2}}\lambda^{-\frac{\delta}{2}}
\eqne
for all $t\in [0,T]$.
\end{lemma}

\begin{proof}
We first define the flow maps between time $s$ and time $t$ (in polar coordinates)
\[
\begin{split}
\p_t \phi (r,\alpha ,s, t) &= (v[\omega ]\circ \phi) (r,\alpha ,s, t),\\
\phi (r,\alpha ,s, s)&=(r\cos(\alpha),r\sin(\alpha)),\\
\p_t \overline{\phi }(r,\alpha ,s, t) &= (v[A \omega ]\circ  \overline{\phi } ) (r,\alpha ,s, t),\\
\overline{\phi} (r,\alpha ,s, s)&=(r\cos(\alpha),r\sin(\alpha)).\\
\end{split}
\]

Note that this definition allows for any $s,t\in \R$, but we will only be concerned with $t\in[0,s]$.

We will denote the polar coordinates of $\phi$ by $\phi_{r},\phi_{\alpha}$ (and analogously for $\overline{\phi}$), so that in particular, when we consider $\omega_{osc}$ in polar coordinates 
$$\omega_{osc}(r,\alpha,s)=\omega_{osc}(\phi_{r}(r,\alpha,s,0),\phi_{\alpha}(r,\alpha,s,0),0),$$ 
$$\overline{\omega_{osc}}(r,\alpha,s)=\omega_{osc}(\overline{\phi}_{r}(r,\alpha,s,0),\overline{\phi}_{\alpha}(r,\alpha,s,0),0).$$

We first note that, since $v[A \omega ]$ has no radial part,
\[
\overline{\phi }_r (r,\alpha ,s, t)=r
\]
for all $t$. On the other hand, for $\phi_r$ we can use the $L^\infty$ estimate \eqref{vr_NlogN} on $v_r[\omega ]$ to obtain, for  $r\in ( (4\lambda )^{-1}, 6\lambda^{-1}) $,
\eqnb\label{calc3}
\begin{split}
&|\phi_{r} (r,\alpha,s, t)-r|\leq \int_s^t \| v_r [\omega ] \|_{L^\infty (B(0,6\lambda^{-1})\setminus B(0,(4\lambda )^{-1}) )}  \\
& \leq \int_s^t \left( \| v_r [\omega_{osc} ] \|_{L^\infty (B(0,6\lambda^{-1})}+\| v_r [\omega_{rad,err} ] \|_{L^\infty (B(0,6\lambda^{-1})\setminus B(0,(4\lambda )^{-1}) )} \right)  \\
& \lec  \lambda^{-\beta} N^{-1-\beta }{\log N} + \| \omega_{rad,err}  \|_{L^2}\\
&\lec \lambda^{-\beta} N^{-1-\beta }{\log N}
\end{split}
\eqne
for all $s,t\in [0,T]$, where we used the fact that  $v_r [\overline{\omega_{rad}} ]=0$ in the second line (recall \eqref{notation_omega_errors} that $\omega_{rad,err}\coloneqq \omega_{rad} - \overline{\omega_{rad}}$) and we used \eqref{w_decomp} in the last line.

As for the angular component we have 
\eqnb\label{angular}
\p_t  (\bar{\phi}_{\alpha}-\phi_{\alpha})(r,\alpha,s,t)=\frac{v_{\alpha}[A \omega - \omega ]\circ \phi}{\phi_{r}(r,\alpha,s,t)} +\frac{v_{\alpha}[A \omega ]\circ \overline{\phi }}{\overline{\phi}_{r}(r,\alpha,s,t)}  -\frac{v_{\alpha}[A \omega ] \circ {\phi }}{\phi_{r}(r,\alpha,s,t)} .
\eqne
Noting that $A \omega -\omega$ has $\alpha$-average zero, we can use Lemma~\ref{lem_v_loglip} in the same way as in \eqref{calc2} to obtain that 
\[\begin{split}
\left|v_{\alpha}[A \omega - \omega ] \right| &\leq \left|v_{\alpha}[A \omega_{osc} - \omega_{osc} ] \right| + \left|v_{\alpha}[A \omega_{rad,err} - \omega_{rad,err} ] \right|    \\
&\lec \| \omega_{osc} \|_{L^\infty} \frac{\log N}{\lambda N} + \|  \omega_{rad,err} \|_{L^2}\\
&\lec \lambda^{-\beta } N^{-1-\beta } {\log N}
\end{split}\]
for $r\in ( (4\lambda )^{-1}, 6\lambda^{-1})$, where we used the fact that $A\overline{\omega_{rad}}-\overline{\omega_{rad}}=0$ in the first line and \eqref{w_decomp} in the last line.

Moreover, since $v_\alpha [A \omega ]$ does not depend on $\alpha$ we have
\[\begin{split}
&|v_{\alpha}[A \omega ] \circ \overline{\phi } -v_{\alpha}[A \omega ] \circ {\phi } |  \leq |v_{\alpha}[A \omega ]( \overline{\phi }_r,0) -v_{\alpha}[A \omega ]( {\phi_r },0) |\\
&\leq \| v[A \omega ]\|_{W^{1,\infty }( B(0,6\lambda^{-1})\setminus B(0,(4\lambda )^{-1}) )} | \phi_r - \overline{\phi}_r |\\
&\lec \left( \| \omega_{osc} \|_{L^{\infty}} \log \| \omega_{osc} \|_{C^1} + \| v [\omega_{rad} ] \|_{W^{1,\infty }( B(0,6\lambda^{-1})\setminus B(0,(4\lambda )^{-1}) )} \right)\frac{\log N}{\lambda^{\beta} N^{1+\beta }}\\
&\lec \left( \lambda^{1-\beta } N^{-\beta } \lambda^{1-\beta+\delta/2} + \lambda^{1-\beta} \right) \lambda^{-\beta} N^{-1-\beta }{\log N}\\
& \lec \lambda^{1-2\beta } N^{-1-\beta } \log N
\end{split}
\]
 for $r\in ( (4\lambda )^{-1}, 6\lambda^{-1}) $, where we used \eqref{v_C1_est} and \eqref{calc3} in the 3rd line, \eqref{w_decomp},\eqref{win_C1} and the support separation \eqref{supps_localized} in the fourth line. 

Finally, we have that, for $r\in(\frac{1}{4\lambda},\frac{6}{\lambda})$ 
\[\begin{split}\frac{v_{\alpha}[A \omega ] \circ {\phi }}{\phi_{r}(r,\alpha)}-\frac{v_{\alpha}[A \omega ] \circ {\phi }}{\overline{\phi}_{r}(r,\alpha)}&\lec \lambda^{2}\lambda^{-\beta}N^{-1-\beta}\lambda^{-\beta} \log N \\
&\lec  \lambda^{2-2\beta} N^{-1-\beta}\log N .
\end{split}
\]
Thus, combining these bounds with \eqref{angular} gives that 
\eqnb\label{angular1}
\left|\p_t  (\overline{\phi}_{\alpha}-\phi_{\alpha}) \right|\lec \lambda^{2-2\beta } N^{-1-\beta } \log N
\eqne
for all $s,t\in [0,T]$, and in particular
$$|\overline{\phi}_{\alpha}(r,\alpha,s,t)-\phi_{\alpha}(r,\alpha,s,t))|\lec  \lambda^{2-2\beta } N^{-1-\beta } \log N.$$
Thus, since 
\[\begin{split}
\omega_{osc}(r,\alpha,s)&=g(\lambda \phi_{r}(r,\alpha,s,0)) \lambda^{1-\beta} N^{-\beta} \cos(N\phi_{\alpha}(r,\alpha,s,0)),\\
\overline{{\omega}_{osc}}(r,\alpha,s)&=g(\lambda r) \lambda^{1-\beta} N^{-\beta} \cos(N\bar{\phi}_{\alpha}(r,\alpha,s,0)),
\end{split}
\]
we can apply both \eqref{calc3} and \eqref{angular1} to obtain
\[
\begin{split}
    |\omega_{osc}-\overline{{\omega}_{osc}}|&\lec \frac{\lambda^{1-\beta}}{N^{\beta}}  N \cdot \lambda^{2-2\beta } N^{-1-\beta } \log N+  \frac{\lambda^{1-\beta}}{N^{\beta}} \lambda \cdot\lambda^{-\beta} N^{-1-\beta }{\log N}\\
    &\lec\frac{\lambda^{1-\beta}}{N^{\beta}}\frac{\lambda^{2-2\beta}\log N}{N^{\beta}}.
\end{split}\]
Integrating over the support of $\omega_{osc}-\overline{\omega_{osc}}$ we get

$$\|\omega_{osc}-\overline{\omega_{osc}}\|_{L^{2}}\lec \|\omega_{osc}\|_{L^2}\frac{\lambda^{2-2\beta}\log N}{N^{\beta}},$$
which proves \eqref{inner_error}, as required.
\end{proof}

\subsection{Step 4 - $H^s$ norm inflation}\label{sec_step4}

Here we finish the proof of Theorem~\ref{illpos}. Namely, we show that for sufficiently large $\lambda$ the only solution to the $2$D Euler equations \eqref{vorticity} with initial conditions $\omega_0$ given by \eqref{initial_data} satisfies 
\eqnb\label{toshow_4.4}
\| \omega \|_{H^{\beta'}} \geq K \qquad \text{ for } t\in[1/T,T].
\eqne

\begin{rem}
Note that since $\|\omega\|_{L^{2}}\leq 1$, if $K>1$ then $\|\omega\|_{H^{s}}\geq K$ for $s\geq \beta'$. Furthermore, note since $\|\omega_0 \|_{H^{\beta}}\leq 1$, independently of $\lambda>0$, we are showing strong ill-posedness in $H^{\beta}$ by considering $\epsilon\omega$ for small $\epsilon>0$. 
\end{rem}

In order to see \eqref{toshow_4.4}, we first recall that the energy conservation and the form \eqref{initial_data} of initial data give
\eqnb\label{osc_l2_final}
\| \omega_{osc} \|_{L^{2}} = \| \overline{\omega_{osc}} \|_{L^2} = \frac{1}{2} \| g \|_{L^{2}} (\lambda N)^{-\beta }.
\eqne

We set $\gamma \coloneqq (\beta_\delta + \beta' )/2$. By interpolation and Lemma~\ref{winerror}
\eqnb\label{osc_error}
\begin{split}
\| \omega_{osc,err} \|_{H^\gamma }^{\beta'} &\lec \| \omega_{osc,err} \|_{L^{2}}^{\beta'-\gamma }  \| \omega_{osc,err} \|_{H^{\beta'}  }^{\gamma } \\
& \lec \| \overline{\omega_{osc}} \|_{L^{2}}^{\beta'-\gamma } \lambda^{-\delta (\beta'-\gamma )/2 }  \left( \| \omega_{osc}\|_{H^{\beta'}} + \| \overline{\omega_{osc}}\|_{H^{\beta'}}  \right)^{\gamma} \\
&\lec   \left( \lambda^{-\beta -\frac{\delta}2 } N^{-\beta } \right)^{\beta'-\gamma } \left( \| \omega_{osc}\|_{H^{\beta'}}^{\gamma }
+ \| \overline{\omega_{osc}}\|_{H^{\beta'}}^{\gamma } \right).    
\end{split}
\eqne

We now show that, for some $\eta>0$,
\eqnb\label{neg_sob}
\|\overline{\omega_{osc}}\|_{\dot{H}^{-\eta}}\lec \|\overline{\omega_{osc}}\|_{L^{2}}(N\lambda^{(2-\beta)})^{-\eta}
\eqne
for $t\in [1/T,T]$. To this end, we first recall the definition \eqref{pseudosol_defs} of $\overline{\omega_{osc}}$,
\[
\overline{\omega_{osc}} =   \lambda^{1-\beta } g(\lambda r) N^{-\beta }\cos \left( N\left( \alpha-\frac{1}{r}\int_{0}^{t}  v_{\alpha}\left[ A\omega \right]  \d s\right) \right).
\]
In order to apply Lemma~\ref{lem_H-delta}, we note that 
\eqnb\label{c1_est0}
\| v [\omega_{osc} ] \|_{L^\infty (B (0,4/\lambda ) )} \lec \| \omega_{osc} \|_{L^\infty} \int_{B(0,8/\lambda )} |y |^{-1} \d y \lec  (N\lambda )^{-\beta } ,
\eqne
and we use  the $C^1$ velocity estimate \eqref{v_C1_est} together with the $C^1$ estimate \eqref{win_C1} of $\omega_{osc}$ to obtain that 
\eqnb\label{c1_est1}
\begin{split}
\| v [\omega_{osc} ] \|_{C^1} &\lec \| \omega_{osc} \|_{L^\infty} \log \| \omega_{osc} \|_{C^1} \\
&\lec \lambda^{1-\beta } N^{-\beta } \lambda^{1-\beta} \log (\lambda^{2-\beta}N^{1-\beta}) \leq \lambda^{1-\beta-\delta}, 
\end{split}
\eqne
recall \eqref{lambda_vs_N} for the relation between $\lambda$ and $N$.
Moreover, since the supports of $\omega_{rad,err}$ and $g(\lambda \cdot )$ are at least $C/\lambda$ apart (recall \eqref{supps_localized}), we can use \eqref{w_decomp} to obtain
\eqnb\label{c1_est2}
\left| v[\omega_{rad,err}] \right|+\left| \nabla  v[\omega_{rad,err}] \right| \lec  (1+\lambda )   \| \omega_{rad,err} \|_{L^2} \leq \lambda^{-\beta-\delta} \quad \text{ in }\supp\, \overline{\omega_{osc}} .
\eqne
Thus, given $t\in [1/T,T]$, letting $W (r),G(r),G_{err}(r) $ be defined by
\[
\begin{split}
W (\lambda r) &\coloneqq \overline{\omega_{osc}} ( r,t),\\
 G(\lambda r) &\coloneqq   (\lambda r)^{-1}  v_{\alpha } [f] ( \lambda r ),\\
  G_{err} (\lambda r) &\coloneqq - N r^{-1} \int_0^t v_{\alpha } [A(\omega_{osc}+\omega_{rad,err})] ( r,s ) \d s ,
\end{split}
\]
we observe that
\[
W = \lambda^{1-\beta } N^{-\beta } g \,\cos \left( N \alpha -Nt \lambda^{1-\beta }  G + G_{err} \right) .
\]
Hence, since $W \in C^2_c ([1/2,4])$ we recall \eqref{c1_est0}--\eqref{c1_est2} to obtain
\[
\| G_{err} \|_{C^1 ([1/2,4])} \leq C N t \lambda^{1-\beta -\delta} \leq (Nt \lambda^{1-\beta })^{1-\sigma } 
\]
for some small $\sigma>0$, where we used the fact that $t\in [1/T,T]$ and $\lambda $ is sufficiently large (and depends on $T$). This lets us use Lemma~\ref{lem_H-delta} to obtain that
\[
\| W \|_{\dot H^{-\eta}} \leq C (N\lambda^{1-\beta } )^{-\eta }\lambda^{1-\beta} N^{-\beta },
\]
where the Sobolev norm is considered on $\R^2$, treating $W$ as a~radial function.
 This, together with  the fact that $\| W \|_{\dot H^{-\eta}} = \lambda^{1+\eta } \| \overline{\omega_{osc} } \|_{\dot H^{-\eta}}$, gives \eqref{neg_sob}, as required.

The upper bound \eqref{neg_sob}, together with the $L^2$ conservation of $\overline{\omega_{osc}}$, let us use Sobolev interpolation (of $L^2$ in terms of $\dot H^{-\eta }$ and $\dot H^s$) to obtain a~lower bound for  $\| \overline{\omega_{osc}} \|_{\dot H^s}$ for $s\in (0,1]$. On the other hand, a~direct calculation shows that  $\|\overline{\omega_{osc}}\|_{\dot H^{1}}\leq CN\lambda^{2-\beta}\|\overline{\omega_{osc}}\|_{L^2}$ for $t\in [1/T,T]$, and so we can interpolate $\dot H^s $ between $L^2$ and $\dot H^1$ to obtain an upper bound on $\| \overline{\omega_{osc}} \|_{\dot H^s}$. Altogether we obtain 
\begin{equation}\label{interapprox}
    \|\overline{\omega_{osc}}\|_{H^{s}}\approx \|\omega_{osc}\|_{L^{2}} (N\lambda^{2-\beta})^{s} \qquad \text{ for each } s\in [0,1].
\end{equation}

Applying interpolation again for $\omega_{osc}$ we obtain
\[
\begin{split}
\| \omega_{osc} \|_{H^{\beta'}}^\gamma &\geq  \frac{\| \omega_{osc} \|_{H^{\gamma } }^{\beta'} }{\| \omega_{osc} \|_{L^{2}}^{\beta'-\gamma }  }\\
& \geq   \frac{1}{\| \omega_{osc} \|_{L^{2}}^{\beta'-\gamma }  }\left( C \| \overline{\omega_{osc}} \|_{H^{\gamma } }^{\beta'}- C\| \omega_{osc,err} \|_{H^{\gamma  } }^{\beta'}  \right)\\
&\geq  C \| \overline{\omega_{osc}} \|_{H^{\beta' } }^{\gamma }  - C\lambda^{-\delta(\beta'-\gamma )/2}  \left(  \| \omega_{osc}\|_{H^{\beta'}}^{\gamma }
+ \| \overline{\omega_{osc}}\|_{H^{\beta' }}^{\gamma   }  \right),
\end{split}
\]
where, in the last inequality, we used \eqref{interapprox} twice (with $s=\gamma$ and with $s=\beta'$) to estimate $\| \overline{\omega_{osc}} \|_{H^\gamma}$, as well as  \eqref{osc_error} to bound $\| \omega_{osc,err} \|_{H^\gamma}$ from above. 
Since the $\| \omega_{osc} \|_{H^{\beta' }}$ norm on the right-hand side can be absorbed by the left-hand side, and the last  $\| \overline{\omega_{osc}}\|_{H^{\beta' }}$ norm is negligible in comparison with the first term on the right-hand side, we thus obtain that
\[
\| \omega_{osc} \|_{H^{\beta'}} \geq C \| \overline{\omega_{osc}} \|_{H^{\beta'}}
\]
for each $t\in [1/T,T]$. Hence, applying \eqref{interapprox} again with $s=\beta'$ we obtain
\eqnb\label{growth_needed}
\|{\omega_{osc}}\|_{H^{\beta'}}\geq C \frac{(N\lambda^{2-\beta})^{\beta'}}{(\lambda N)^{\beta}}\geq C \lambda^{\tilde{\epsilon}},
\eqne
where  $\tilde{\epsilon}>0$ is a~small constant. Thus, choosing sufficiently large $\lambda$ shows growth of $\| \omega_{osc} \|_{H^{\beta'}}$, and hence also of $\| \omega \|_{H^{\beta'}}$, due to the localization \eqref{supps_localized} and  \eqref{ss_consequence}. In particular we obtain \eqref{toshow_4.4}, as required.

\section{Gluing: Loss of regularity}\label{gluing}

Here we prove Theorem~\ref{lossreg}, namely we  show  existence of a~solution that loses regularity instantly, and furthermore it is the unique classical solution (as in Definition~\ref{classol}) and it is global in time. \\

By rescaling the initial data we can assume that $\epsilon = 1$; thus, given $\beta \in (0,1)$, we need to find $\omega (x,0)$ such that there exists a~unique global classical solution to 2D Euler (as in Definition~\ref{classol}) with this initial condition which satisfies
\[
\|\omega(x,0)\|_{H^{\beta}}\leq 1 ,
\]
\eqnb\label{to_show_loss}\|\omega(x,t)\|_{H^{\beta'}}=\infty\quad  \text{ for }\quad  t \in(0,\infty),\ \beta'>\frac{(2-\beta)\beta}{2-\beta^2}.\vspace{1cm}
\eqne

First, given $j\geq 1$, we will denote by $\omega_{j}(x,t)$ a~smooth solution to 2D Euler equation given by Theorem~\ref{illpos} such that

\begin{equation}\label{Ngrowth}
   \|\omega_{j}(x,t)\|_{H^{s}}\geq 4^{j}\ \text{ for }\ s>\frac{(2-\beta)\beta}{2-\beta^2}+\frac{1}{j},\ t\in\left[ \frac{1}{4^{j}},1 \right] . 
\end{equation}
Note also that, by construction, we can choose the $\omega_{j}$ so that, for all $t\in [0,1]$,
\eqnb\label{supp_pieces}
|\supp \,\omega_{j}|\leq 2^{-j},\, \,\supp\, \omega_{j}\subset B(0,1).
\eqne
Moreover, setting  $p\coloneqq  2/(1-\beta )$ we can assume that  $\| \omega_j (\cdot ,t)\|_{L^p}=C $ for all $t\geq 0$, where $C>0$ is a~constant, by the $L^p$ conservation and the form \eqref{initial_data} of the initial data.

We will consider initial conditions $\omega (x,0)$ of the form
\eqnb\label{initial_data_gluing}
\omega(x,0) \coloneqq \sum_{j=1}^{\infty}T_{R_{j}}\bigg(\frac{\omega_{j}(x,0)}{2^{j}}\bigg),
\eqne
where $T_{R}(f(x_{1},x_{2}))=f(x_{1}-R,x_{2}).$ 
For brevity, we will use the notation
\[
\wo_j (x,t) \coloneqq T_{R_j} \left( \frac{\omega_j (x,t/2^j )}{2^j}\right),
\]
where the $R_j$'s remain to be fixed. Some properties to keep in mind are:
\begin{itemize}
    \item $\wo_j$ is a~smooth global solution to 2D Euler with    \eqnb\label{wo_j_inflation}
    \| \wo_j (\cdot , t) \|_{H^s} \geq 2^j \qquad \text{ for }s>\frac{(2-\beta )\beta }{2-\beta^2 } + \frac1j,\,\, t\in [2^{-j},2^j ],
\eqne    
     due to \eqref{Ngrowth}.
    Furthermore, we have
\eqnb\label{omega_lp}
    \begin{split}
    \left\| \wo_j (x,0) \right\|_{H^{\beta}}&\leq 2^{-j} ,\qquad     \left\| \wo_j (x,0)  \right\|_{L^{1}\cap L^p}\leq C 2^{-j},
    \end{split}
    \eqne
    where we set $\| \cdot \|_{L^1\cap L^p } \coloneqq \| \cdot \|_{L^1} + \| \cdot \|_{L^p} $ (recall $p=2/(1-\beta)>2$), as well as 
    \eqnb\label{supp_omega_j_tilde}
    |\supp\, \tilde{\omega}_j | \leq 2^{-j}, \supp\,\tilde{\omega}_j \subset B((R_j,0),1) \qquad \text{ for }t\in [0,2^j].
    \eqne
    \item Given the truncated initial conditions
    \begin{equation}\label{trinitial}
        \sum_{j=1}^{J}\wo_j (x,0) ,
    \end{equation}
        we will refer to the unique global-in-time  solution to the 2D Euler equations \eqref{vorticity}--\eqref{bs_law} with initial condition \eqref{trinitial}  as $\omega_{tr,J}$, and, for any $t\in[0,T]$, $m,J\in\N$, there exists a~constant $C_{m,J,T}$, independent of the choice of $(R_{j})_{j\in\N}$ such that
   \eqnb\label{Hm_bounds_tr_and_wo}
    \|\omega_{tr,J} (\cdot ,t )\|_{H^{m}}\leq C_{m,J,T},\qquad  \left\|\wo_J (\cdot , t) \right\|_{H^{m}}\leq C_{m,J,T}. 
    \eqne
    \item Furthermore, noting that $\| v [f ]\|_{L^\infty } \leq C_q (\| f\|_{L^1}+\| f \|_{L^q})$ for any $f$ and $q>2$ we deduce from \eqref{omega_lp} that
    $$\|v[ \omega_{tr,J}] \|_{L^{\infty}}, \| v[\wo_J]\|_{L^\infty }\leq v_{max}$$
    for all $t\geq 0$, where $v_{max}$ is some constant independent of $J$ and of the choice of $(R_{j})_{j\in\N}$. We also deduce from \eqref{Hm_bounds_tr_and_wo} that
\eqnb\label{vel_decay_pieces}
\begin{split}
| \nabla^k v[\omega_{tr,J}](x,t) | &\leq  \frac{C_{k, J, T}}{(1+\mathrm{dist} \, (x,\supp\, \omega_{tr,J}))^{k+1}},\\
 | \nabla^k v[\wo_J (x,t) ]| &\leq  \frac{C_{k, J, T}}{(1+\mathrm{dist} \, (x,\supp\, \wo_J ))^{k+1}}
 \end{split}
\eqne    
    for all $x\in \R^2$, $t\in [0,T]$.
    
    \item Moreover, 
    \eqnb\label{supp_omega_tr}
    |\supp \, \omega_{tr,j}|\leq 1-2^{-j}\leq 1\qquad \text{ for all }j\geq 1,
    \eqne
   as a~property of the $2$D Euler equations.    
\end{itemize}

We will define $R_{1}=0$, $R_{j+1}=D_{j+1}+D_{j}+R_{j}$ and show that if $D_{j}>0$ are big enough, then there exists a~global solution $\omega_\infty$ with loss of regularity. \\

We first construct $\omega_\infty$ as a~limit of $\omega_{tr,j}$ as $j\to \infty$. To this end, for any fixed $(D_{j})_{j=1,...,J}$, we define inductively the ``$J+1$-th approximation'' by

$$\overline{\omega}_{tr,J+1}\coloneqq \omega_{tr,J}+\wo_{J+1} .$$
 It fulfils an evolution equation of the form

$$\p_t  \overline{\omega}_{tr,J+1}+v[\overline{\omega}_{tr,J+1}]\cdot\nabla \overline{\omega}_{tr,J+1}+F=0,$$
where
$$F\coloneqq -v[\omega_{tr,J}]\cdot\nabla \wo_{J+1}-v\left[ \wo_{J+1}\right] \cdot\nabla\omega_{tr,J},$$
and, since, for $t\in[0,2^{J+1}]$, 
$$\mathrm{dist} \left( \supp \, \omega_{tr,J},\supp\, \wo_{J+1}\right)\geq D_{J+1}-2v_{max}2^{J+1}-2,$$
the $H^m$ boundedness of the vorticity functions \eqref{Hm_bounds_tr_and_wo} and the decay \eqref{vel_decay_pieces} of the corresponding velocity fields $v$ implies that 
\eqnb\label{F_vanishes}
\|F(\cdot , t)\|_{H^{4}} \leq  C_{J} (D_{J+1}-2v_{max}2^{J+1}-2)^{-1} \to 0\qquad \text{ as }D_{J+1} \to \infty,
\eqne
uniformly in  $t\in[0,2^{J+1}]$. \\

We set 
\[
W_{J+1}\coloneqq \omega_{tr,J+1}-\overline{\omega}_{tr,J+1},
\]
and we use  the evolution equation for $W_{J+1}$, 
\[
\p_t W_{J+1} + v[W_{J+1}]\cdot \nabla W_{J+1} + v[W_{J+1}]\cdot \nabla \overline{\omega }_{tr,J+1}+v\left[ \overline{\omega }_{tr,J+1} \right] \cdot \nabla W_{J+1} = F
\]
 to obtain that
$$\frac{\d \|W_{J+1}\|_{H^{4}}}{\d t}\leq C(\|W_{J+1}\|^{2}_{H^4}+\|W_{J+1}\|_{H^4}\|\overline{\omega}_{tr,J+1}\|_{H^{5}}+\|F\|_{H^{4}}),$$
where we used the velocity estimates
 \[
 \begin{split}
 \| v [W_{J+1}] \|_{H^4 (\supp\, \overline{\omega}_{tr,J+1} )} &\lec \| v [W_{J+1}] \|_{L^\infty (\supp\, \overline{\omega}_{tr,J+1} )} +  \|D^4 v [W_{J+1}] \|_{L^2 } \\
 &\leq \| W_{J+1} \|_{H^3}  ,\\
 \| v [\overline{\omega}_{tr,J+1} ] \|_{C^4 (\supp\, W_{J+1} )} &\lec \| \overline{\omega}_{tr,J+1} \|_{H^7},
 \end{split}
 \]
  as well as the facts that $|\supp \, \overline{\omega}_{tr,J+1} |\leq 1$ and $|\supp\, W_{J+1}|\leq 2$, due to \eqref{supp_omega_tr}. 
 Thus, since $W_{J+1} (\cdot ,0)=0$, and $\|\overline{\omega}_{tr,J+1}\|_{H^{7}}\leq \|{\omega}_{tr,J}\|_{H^{7}} +\|\wo_{J+1} \|_{H^{7}}\leq C_J$  for $t\in[0,2^{J+1}]$ (due to \eqref{Hm_bounds_tr_and_wo}) and since $F$ vanishes in the limit $D_{J+1}$ (recall \eqref{F_vanishes}), we can find 
\eqnb\label{DJ}
\widetilde{D_{J+1}}\geq  4^{J+1}(v_{max}+1)+2
\eqne 
such that 
\begin{equation}\label{errorh4}
    \left\|\omega_{tr,J+1}(\cdot ,t)-\omega_{tr,J}(\cdot ,t)-\wo_{J+1}(\cdot ,t) \right\|_{H^{4}}\leq 2^{-J-1}
\end{equation}
for $D_{J+1}\geq \widetilde{D_{J+1}} $ and all $t\in [0,2^{J+1}]$.\\

 Given $a,d>0$ we denote by 
\[
K=\overline{B(0,d)}\times[0,a]
\]
an arbitrary compact set in space-time. Note that, for each such $K$ the support of $\widetilde{\omega}_{J+1}$ is disjoint with $K$ for sufficiently large $J$. Thus \eqref{errorh4} implies that $\{ \omega_{tr,J} \}_{J\geq 1} $ is Cauchy in $C^0_t H^4_x (K)$, and so there exists $\omega_\infty \in C ([0,\infty ) ; H^4_{loc} (\R^2) )$ such that  
\eqnb\label{limit_in_Linfty_H4}
 \|  \omega_{tr,J} - \omega_{\infty} \|_{C^0_t H^4_x (K)}\to 0\qquad \text{ as } J\to \infty 
 \eqne
 for every $K$.  Note that in particular $\omega_{\infty}\in C^0_t C_{x}^{2}(K)$, and so, since $ \omega_\infty \in C^0 ([0,a];L^1 (\R^2 )) $ (a~consequence of \eqref{errorh4} and \eqref{omega_lp}) we see that, for each $K$,  $D^\alpha v [\omega_\infty ]$ exists at each point of  $ K$ and each multiindex $\alpha$ with $|\alpha | \leq 2$, and 
 \eqnb\label{comminfty}
 \begin{split}
    \|  &v[\omega_{\infty} ]-  v[\omega_{tr,J} ] \|_{C^0_t C^2_x (K) }  \\
    & \leq C_{K}\left( \| \chi_{B(0,2d)^c}(\omega_\infty - \omega_{tr,J}) \|_{C^0 ([0,a];L^1 )} + \| \omega_\infty - \omega_{tr,J} \|_{C^0 ([0,a];C^2 ( B(0,2d)))} \right)  \\
    &\leq C_K \sum_{j\geq J } \left\| \chi_{B(0,2d)^c}(\omega_{tr,j+1} - \omega_{tr,j})   \right\|_{C^0 ([0,a];L^1 )} + o(1)\\
   &\leq C_K \sum_{j\geq J } \left( 2^{-j} + \| \wo_{j+1} \|_{C^0 ([0,a];L^1)} \right) + o(1)\\
   &\leq o(1)
\end{split}
\eqne
 as $J\to \infty$, where we used the Biot-Savart law \eqref{bs_law} in the first inequality, \eqref{limit_in_Linfty_H4} in the second inequality, \eqref{errorh4} in the third and \eqref{omega_lp} in the fourth.\\

Having found the limit $\omega_\infty$ with convergence properties \eqref{limit_in_Linfty_H4},  \eqref{comminfty}, we can now take the limit $J\to \infty$ in the weak formulation of $\p_t \omega_{tr,J}+ v[\omega_{tr,J}]\cdot \nabla \omega_{tr,J}=0$ (which is obtained by multiplying by a~smooth function that is compactly supported in $K$, and integrating)  to obtain that $\p_t \omega_\infty = - v[\omega_\infty ]\cdot \nabla \omega_{\infty} \in C^0_t C^1_x (K)$. In particular $ \omega_{\infty } \in C^1_{x,t} (K)$,  which gives that $\omega_\infty$ is a~classical solution of the Euler equations in the sense of Definition~\ref{classol}.\\

We now show that $\omega_\infty$ instantly loses regularity. Namely we show \eqref{to_show_loss}, for which it is sufficient to consider only $s\in \left( \frac{(2-\beta )\beta }{2-\beta^2} , 1\right)$. Given such $s$, and $\tt >0$ we fix $J\geq 1$ such that
\eqnb\label{choice_of_J}
s> \frac{(2-\beta )\beta }{2-\beta^2} +\frac1J \qquad \text{ and } \qquad 2^{J+1}\geq \tt.
\eqne
 Using the short-hand notation
\eqnb\label{shorthand_Bj}
B_j \coloneqq B((R_j,0),D_j),
\eqne
we obtain 
\eqnb\label{inflation_gluing}
\begin{split}
  \|\omega_{\infty}(\cdot,\tt )\|_{H^{s}}&\geq\left\|\sum_{j\geq {J+1}}  \wo_{j} (\cdot ,\tt ) \right\|_{H^{s}} -\|\omega_{tr,J}\|_{H^{s}}\\
  &\hspace{2cm}-\sum_{j\geq J}\left\|\omega_{tr,j+1}(\cdot ,\tt )-\omega_{tr,j}(\cdot ,\tt )-\wo_{j+1} (\cdot ,  \tt ) \right\|_{H^{s}}\\
    &\geq \left\|\sum_{j\geq J+1} \wo_j (\cdot , \tt )  \right\|_{\dot H^s}- C_{s,\tt} \\ 
    &\geq \left\|\wo_k (\cdot , \tt )  \right\|_{\dot H^s}- C_{s,\tt} \geq 2^k - C_{s,\tt }
\end{split}
\eqne
for any $k\geq J+1$, where  we used \eqref{errorh4} in the second inequality, as well as \eqref{ss_consequence} and the fact that $\supp\, \wo_j (\cdot , \tau  ) \subset B((R_j,0),1)$ for all $j\geq J+1$ (recall \eqref{supp_pieces}) in the third inequality.
 Since $k\geq J+1$ is arbitrary, we obtain \eqref{to_show_loss}, as required. \\

In order to show that $\omega_\infty$ is the unique solution in the sense of Definition~\ref{classol}, we first denote by $\phi(x,t)$ the flow map of $\omega_{\infty}$, and we set
\eqnb\label{gluing_uniq_traj}
\begin{split}
\omega_{\infty,j}(x,t)&\coloneqq \wo_{j}(\phi^{-1}(x,t) ,0), \qquad \omega_{\infty,\leq J} \coloneqq \sum_{j=1}^{J} \omega_{\infty,j} .
\end{split}
\eqne
This allows us to decompose $\omega_\infty$ into pieces,
\[
\omega_{\infty}=\sum_{j=1}^{\infty} \omega_{\infty,j},
\]
where each piece satisfies 
\begin{equation}\label{omegajev}
  \p_{t}\omega_{\infty,j}+v[\omega_{\infty}]\cdot \nabla \omega_{\infty,j}=0.
\end{equation}
In particular (recall \eqref{supp_pieces})
\eqnb\label{support_measure_limitj}
| \supp \, \omega_{\infty, j} | \leq 2^{-j}\qquad \text{ for all times }t\geq 0.
\eqne

We now show that, for each fixed $a>0$
\eqnb\label{toshow_winfty_ests}
\|\omega_{\infty,j}\|_{C^1}\leq  \ee^{M_{j}\ee^{\tilde{C}M_{j}a}} \qquad \text{ and }\qquad  \|\omega_{\infty,\leq j}\|_{C^1}\leq  \ee^{S_{j}\ee^{\tilde{C}S_{j}a}}
\eqne
for all $t\in [0,a]$ and all  $j$ such that $2^{j-1}\geq a$, where $\tilde C>1$ is a~universal constant and
\eqnb\label{pieces_c1_bds_uniq}
\begin{split}
M_{j}&\coloneqq  \max \left( 1, \| \wo_j (\cdot , 0)  \|_{C^1} \right) ,\qquad S_{j}\coloneqq \sum_{i=1}^{j}M_{i}.
\end{split}
\eqne

To this end, we first apply the $C^1$ estimate \eqref{C2_est_cons} to $\wo_j$ and $\omega_{tr,j}$, $j\geq 1$, to obtain 
\eqnb\label{temp22}
\|\wo_j \|_{C^1}\leq  \ee^{M_{j}\ee^{\tilde{C}M_{j}a}} \qquad \text{ and }\qquad  \|\omega_{tr,j}\|_{C^1}\leq \ee^{S_{j}\ee^{\tilde{C}S_{j}a}}   
\eqne
for all $j$ and all $t\in [0,a]$, where $\tilde C >1$ is a~constant. Thus, since for $2^{j} \geq a$ 
\eqnb\label{limit_j_remains_supported}
\begin{split}
\omega_{\infty, \leq j}(t)\,\, \text{ and }\,\, \omega_{tr,j} (t) &\text{ remain supported in }B(0,R_j+D_j) , \\
\omega_{\infty,  j}(t)\,\, \text{ and } \,\, \wo_{j} (t) &\text{ remain supported in }B_j 
\end{split}
\eqne
for $t\in [0,a]$ (recall \eqref{shorthand_Bj} and \eqref{DJ}), we obtain that 
\[
\begin{split}
\| \omega_{\infty, j } - \wo_j \|_{H^4 } &= \left\| \omega_\infty - \omega_{tr,j-1} - \sum_{k\geq j} \wo_k \right\|_{H^4(B_j)}\\
&= \left\| \sum_{k\geq j } \left( \omega_{tr,k} - \omega_{tr,k-1}- \wo_{k} \right) \right\|_{H^4 (B_j )}\\
&\leq  \sum_{k\geq j }\left\|\omega_{tr,k}-\omega_{tr,k-1}-\wo_{k}\right\|_{H^{4}}\leq \sum_{k\geq j } 2^{-k} = 2^{-(j-1)}
\end{split}
\] 
for all $t\in [0,a]$ and $j$ such that  $2^{j-1}\geq a$, where we used \eqref{errorh4} in the last line. This and the first claim of \eqref{temp22}
proves the first claim of \eqref{toshow_winfty_ests}, upon possibly taking $\tilde C$ larger. A~similar calculation,
\[
\begin{split}
\| \omega_{\infty, \leq j } - \omega_{tr,j} \|_{H^4 } &= \left\| \omega_\infty - \omega_{tr,j} - \sum_{k\geq j+1} \wo_k \right\|_{H^4(B(0,R_j+D_j) )}\\
&\leq  \sum_{k\geq j+1 }\left\|\omega_{tr,k}-\omega_{tr,k-1}-\wo_{k}\right\|_{H^{4}(B(0,R_j+D_j))}\\
&\leq \sum_{k\geq j+1 } 2^{-k} = 2^{-j}
\end{split}
\] 
for $t\in [0,a]$ and $j$ such that $2^j\geq a$, together with \eqref{temp22} shows the second claim of \eqref{toshow_winfty_ests}, as required.

We emphasize that all of the above claims hold for each choice of the sequence $\{ D_j \}_{j\geq 1}$ satisfying $D_j\geq \widetilde{D_j}$, where $\widetilde{D_J} $ was defined by \eqref{DJ}--\eqref{errorh4}.
We now prove uniqueness of $\omega_\infty$, provided that each $D_j$ is chosen larger, namely that 
\eqnb\label{DJ_actual}
{D_{j}}\geq  \widetilde{D_j} +\exp \left( \exp \left( {2M_{j}\exp \left( \tilde{C}M_{j}2^{j}\right)} \right) \right)  .
\eqne 

Indeed, suppose that there exists another classical solution $\tilde{\omega}_{\infty}$ of the Euler equations with initial data \eqref{initial_data_gluing}, and let $\tilde{\omega}_{\infty,j}$ be defined in the same way as $\omega_{\infty,j}$, but with the flow map given by $\tilde{\omega}_{\infty}$ so that
\begin{equation}\label{tildeomegajev}
    \tilde{\omega}_{\infty}(x,t)=\sum_{j=1}^{\infty} \tilde{\omega}_{\infty,j}(x,t),\quad  \p_{t}\tilde{\omega}_{\infty,j}+v[\tilde{\omega}_{\infty}]\cdot \nabla \tilde{\omega}_{\infty,j}=0.
\end{equation}

 Note that $\tilde{\omega}_{\infty}$ conserves its $L^{p}$ norms with time and in particular, it moves at most with speed $v_{max}$.

We let 
\[
T\coloneqq \sup \{ T'\geq 0 \colon \omega_\infty (t) = \tilde \omega_\infty (t) \quad \text{ for all }t\in [0,T'] \} 
\]
and we set 
\[
 W\coloneqq \omega_{\infty}-\tilde{\omega}_{\infty},  \quad  W_{j}\coloneqq \omega_{\infty,j}-\tilde{\omega}_{\infty,j} \quad \text{ and } \quad W_{\leq j } \coloneqq \sum_{k=1}^j W_k.
\]
Clearly
\eqnb\label{eq_Wj}
\p_t W_j + v[\tilde \omega_\infty ] \cdot \nabla W_j + v[W ]\cdot \nabla \omega_{\infty,j}=0.
\eqne

In order to estimate $v[W ]$ in $L^2$ we fix $j_0$ such that
\[
2^{j_0-1}\geq T+1.
\]
Note that, since $\supp\, W_{\leq j-1} \subset B(0,R_{j-1}+D_{j-1})$ and $\supp\,W_k \subset B_k$ for $k\geq j\geq j_0$ (a~consequence of 
\eqref{limit_j_remains_supported}), and since $|x-y|\geq D_j -2Tv_{max} - 2$ for $x\in B_j$ and $y\in B(0,R_{j-1}+D_{j-1}) \cup \bigcup_{k\geq j+1} B_k$
we have 
\[\begin{split}
\| v[W] \|_{L^2 (\supp \, \omega_{\infty,j} )} &\lec \left( \int_{\supp \, \omega_{\infty,j}} \left( \int_{\supp \, W_j} \frac{|W_j(y)|}{|x-y|} \d y \right)^2 \d x \right)^{1/2} \\
&\hspace{3cm}+ \frac{ \| W \|_{L^2} }{D_j -2(T+1)v_{max} - 2}\\
&\lec \| W_j \|_{L^2} + \frac{ \| W \|_{L^2} }{D_j -2(T+1)v_{max} - 2}
\end{split}
\]
for each $t\in [T,T+1]$, where we used \eqref{support_measure_limitj}, as well as the fact that $1=\chi_{B(y,2)}(x)$ under the first integral and Young's inequality $\| f \ast g\|_2 \leq \| f \|_2 \|g \|_1$ in the last line. Thus multiplying \eqref{eq_Wj} by $W_j$ and integrating we obtain the energy estimate
\eqnb\label{l2_uniq_setup1}
\begin{split}
\frac{\d }{\d t}\|W_{j}\|_{L^{2}}&\leq C \|W_{j}\|_{L^2}\|\omega_{\infty,j}\|_{C^{1}}+\frac{\|W\|_{L^{2}}}{D_{j}-2(T+1) v_{max}-2}\|\omega_{\infty,j}\|_{C^{1}}\\
&\leq C \|W_{j}\|_{L^2}\|\omega_{\infty,j}\|_{C^{1}}+\ee^{-\ee^{M_{j}\ee^{\tilde{C}M_{j}2^{j}}}}F
\end{split}
\eqne
for $t\in [T,T+\epsilon ]$, $j\geq j_0$, where  $\epsilon \in (0,1)$,
\[
U(t) \coloneqq \sup_{s\in [T,t]} \| W(\cdot ,s)\|_{L^2} \quad \text{ for }t\in [T,T+\epsilon ],
\]
and we used the lower bound on $D_j$ \eqref{DJ_actual} (recall also \eqref{DJ}); note that the factor of $2$ in \eqref{DJ_actual} is used to absorb the upper bound \eqref{toshow_winfty_ests} on $\| \omega_{\infty,j } \|_{C^1}$. Thus, using the upper bound \eqref{toshow_winfty_ests} again, the ODE fact \eqref{ode_fact} shows that
\eqnb\label{l2_uniq_temp}
\|W_{j}\|_{L^2}\leq\epsilon \, \ee^{C\epsilon \ee^{M_{j}\ee^{\tilde{C}M_{j}2^{j}}}}\ee^{-\ee^{M_{j}\ee^{\tilde{C}M_{j}2^{j}}}}U
\eqne
for all $t\in [T,T+\epsilon)$, $j\geq j_0$.
Hence, taking $\epsilon \in (0,1)$ small enough so that for each $j\geq j_0$ the product of the two exponential functions above is bounded by $j^{-2}$, we obtain 
\eqnb\label{Wj_final}\|W_{j}\|_{L^2}\leq \frac{\epsilon}{j^{2}}U
\eqne
for $t\in [T,T+\epsilon]$, $j\geq j_0$. 

As for $j<j_0$ we have  
\eqnb\label{l2_uniq_setup2}
\frac{\d }{\d t}\left\|W_{< j_0 } \right\|_{L^{2}}\leq C \left\|W_{< j_0} \right\|_{L^{2}}\|\omega_{< j_{0}}\|_{C^{1}}+\frac{\|W\|_{L^{2}}}{D_{j_{0}}-2(T+1) v_{max}-2}\|\omega_{< j_{0}}\|_{C^{1}},
\eqne
since $\supp\, W_{< j_0 } \subset B(0,R_{j_0}+D_{j_0})$. Thus, applying the ODE fact \eqref{ode_fact} again, we obtain $\| W_{< j_0 } \|_{L^2} \leq C_T \epsilon U$, where $C_T>0$ depends on $T$. Adding this  inequality to \eqref{Wj_final}, for $j\geq j_0$, gives that $\| W \|_{L^2} \leq (C_T+C)\epsilon U$ for all $t\in [T,T+\epsilon]$, and so, taking $\sup$ gives 
\eqnb\label{l2_uniq_concl}
U(T+\epsilon) \leq (C_T+C)\epsilon U(T+\epsilon ).
\eqne
Taking $\epsilon$ sufficiently small, we thus obtain $U(T+\epsilon)=0$, which proves uniqueness, as required.

\section*{Acknowledgements}
This work is supported in part by the Spanish Ministry of Science
and Innovation, through the “Severo Ochoa Programme for Centres of Excellence in R$\&$D (CEX2019-000904-S)” and 114703GB-100. DC and LMZ were partially supported by the ERC Advanced Grant 788250. WO was partially supported by the Simons Foundation.

\bibliographystyle{alpha}

\end{document}